\documentclass[10pt,reqno]{amsart}
\usepackage[scale=0.75, centering, headheight=14pt]{geometry}

\usepackage{mathtools}

\usepackage{graphicx}
\usepackage{fancyhdr}
\usepackage{amssymb,amsmath,enumerate}
\usepackage{caption}
\usepackage{float}
\usepackage{subcaption}
\usepackage{comment,color}
\usepackage{placeins}
\usepackage[dvipsnames]{xcolor}
\usepackage{amsthm}
\usepackage{thmtools}
\usepackage{bbm}
\newcounter{assump}
\renewcommand{\theassump}{A\arabic{assump}}

\usepackage[colorlinks=true, linkcolor=black, citecolor=black, urlcolor=black]{hyperref}
\usepackage[nameinlink, noabbrev]{cleveref}

\DeclareMathOperator*{\argmin}{arg\,min}

\DeclareMathOperator*{\cof}{cof}
\DeclareMathOperator*{\diver}{div}

\newtheorem{theorem}{Theorem}[section]
\newtheorem{lemma}[theorem]{Lemma}

\newtheorem{remark}[theorem]{Remark}

\newtheorem{prop}[theorem]{Proposition}
\newcommand{\norm}[1]{\lVert#1\rVert}

\newcommand{\R}{\mathbb{R}}
\newcommand{\PP}{\mathbf{P}}
\newcommand{\XX}{\mathbf{X}}
\newcommand{\YY}{\mathbf{Y}}

\newcommand{\NN}{\mathbb{N}}

\newcommand{\dd}{\mathrm{d}}
\newcommand{\eps}{\varepsilon}
\newcommand{\cL}{\mathcal{L}}
\newcommand{\SSS}{\mathbb{S}}
\newcommand{\cR}{\mathcal{R}}
\newcommand{\Ocof}{\Pi_{\ker(\cof(\PP ))}}

\newcommand{\cP}{\mathcal{P}}
\newcommand{\tr}{{\rm tr}}

\begin{document}
\title{Convergence of a least-squares splitting method for the Monge-Amp\`ere equation}

\author[A. Peruso]{Anna Peruso}

\address{
Institute of Mathematics, \'Ecole Polytechnique F\'ed\'erale de Lausanne, 
1015 Lausanne, Switzerland, 
and
Geneva School of Business Administration (HEG-GEN\`EVE), 
University of Applied Sciences and Arts Western Switzerland (HES-SO), 
1227 Carouge, Switzerland, 
}
\email{anna.peruso@epfl.ch, anna.peruso@hesge.ch}

\author[M. Sorella]{Massimo Sorella}
\address{Department of Mathematics, Imperial College London}
\email{m.sorella@imperial.ac.uk}

\subjclass[2020]{ 35Q, 65M12, 65M15 }

\keywords{Monge Amp\'ere equation, least-squares method, alternating projections}

\date{\today}

\begin{abstract}

We study the theoretical convergence of the nonlinear least-squares splitting method for the Monge-Amp\`ere equation introduced in \cite{dean} in which each iteration decouples the pointwise nonlinearity from the differential operator and consists of a local nonlinear update followed by the solution of two sequential Poisson-type elliptic problems. While the method performs well in computations \cite{caboussat,peruso}, a rigorous convergence theory has remained unavailable.

Our main observation is that the iteration admits a reformulation as an alternating-projection scheme on Sobolev spaces $H^m$, $m\ge 0$. At a solution, the G\^ateaux differentials of the projection maps are the linear projections onto the corresponding tangent spaces. We prove that these tangent spaces are transverse, and hence the linearization of the alternating-projection map is a  contraction by classical Hilbert-space theory for alternating projections.
Building on this geometric characterization, we develop two distinct strategies to establish local convergence of the original splitting method. The first applies on the two-dimensional torus $\mathbb{T}^2$ for initial data sufficiently close to a solution $u\in H^4(\mathbb{T}^2)$. The second applies to  smooth bounded and convex domains for  solutions $u\in C^4(\overline{\Omega})$. To the best of our knowledge, this yields the first rigorous convergence result for the splitting method at the continuous level and provides a functional-analytic explanation for its observed numerical robustness.

\end{abstract}

\maketitle

\section{Introduction}
Fully nonlinear, second-order partial differential equations  appear across the sciences and engineering, but both their analysis and numerical approximation remain challenging. While Galerkin-type methods adapt naturally to semilinear and many quasilinear problems, they do not directly extend to fully nonlinear equations: integration by parts cannot be used to transfer derivatives onto test functions and there is no weak formulation in standard Sobolev spaces. Nevertheless, over the past two decades a variety of numerical approaches have been proposed. Much of this development has been driven by the Monge-Amp\`ere equation,
\begin{equation*}
\det D^2 u(x) = f(x,u,\nabla u), \qquad x\in\Omega\subset\mathbb{R}^d,
\end{equation*}
often viewed as a canonical model for fully nonlinear elliptic PDEs \cite{dephilippis}. By contrast, comparatively fewer works address broader classes of fully nonlinear second-order equations.

Most Galerkin-based approaches to fully nonlinear PDEs can be viewed as combining two ingredients, \textit{linearization} and \textit{discretization}, and the order in which they are applied leads to different schemes. A common linearization strategy is Newton's method. In the \textit{linearize-then-discretize} approach, one first approximates the nonlinear PDE by a sequence of linear problems and then applies a Galerkin discretization.  Within this paradigm, the damped Newton iteration is the most common choice \cite{rapetti,lakkis,amireh}, but it is not the only one. Alternative iterative schemes include Poisson-based fixed-point schemes \cite{peraire} and augmented Lagrangian or least-squares formulations that split the nonlinear term from the differential operator \cite{dean,lagrangian}. Despite promising numerical performance, rigorous results on local convergence of the nonlinear iteration are limited.
The only rigorous convergence result at the continuous level for iterative methods for the Monge-Amp\`ere equation, to the authors knowledge, is the one for a damped Newton-type iteration method, proved in \cite{rapetti} and improved in \cite{saumier}.
On the other hand, the \textit{discretize-then-linearize} approach starts from a discrete nonlinear problem and then applies linearization at the algebraic level. This direction has a broader literature \cite{bohmer,brenner,brenner2}, and typically employs Newton's method as the linearization mechanism. A discrete-level analysis can simplify the study of the iteration, but it requires establishing consistency between the discrete and continuous linearizations, as emphasized in \cite{bohmer,brenner}, limiting the choice of the discretization.  Finally, a different family of Galerkin methods avoids the linearize-discretize dichotomy by working with an $\varepsilon$-regularized PDE; see, for instance, the vanishing moment method \cite{neilan} and the more recent developments in \cite{gallistl}. We refer to \cite{feng,neilan-salgado} for recent surveys on numerical methods for the Monge-Amp\`ere equation.

In this work we analyze the nonlinear least-squares method introduced in \cite{dean,caboussat}, and recently revisited in \cite{peruso}. The method targets \(H^2(\Omega)\)-solutions of the Monge-Amp\`ere equation by combining a least-squares variational formulation with an iterative algorithm that decouples the nonlinearity from the differential operator at the continuous level; the resulting sequence of linear problems is then discretized. A practical appeal of this approach is its low per-iteration cost: the nonlinear update is performed pointwise, while the variational step reduces to the solution of two sequential Poisson problems, as proposed in \cite{peruso}. Numerical studies report robust convergence and the framework has been successfully extended to other fully nonlinear models \cite{dimitrios,lsorthogonal,prins,yadav}.

\subsection{Contributions}
We study the convergence of the nonlinear least-squares splitting method introduced in \cite{dean} through the lens of alternating projections in an infinite-dimensional setting, see e.g. \cite{caboussat,peruso} for numerical evidence of convergence. Our main result, \Cref{thm:main}, establishes local linear convergence of the splitting algorithm on the two-dimensional torus $\Omega=\mathbb{T}^2$. The key observation is that the iteration can be interpreted as alternating projections onto a nonconvex set and a linear subspace in Sobolev spaces. To our knowledge, this yields the first continuous-level convergence proof for this splitting approach; see for example \cite{brenner}, where the convergence of this method is highlighted as an open problem. 

We first analyze the scheme when the projections are taken with respect to the $H^m$ topology for Hessian matrices, with $m\ge 2$, see \Cref{sec:regular}. In this simplified setting, the composition of the two projection maps is Fr\'echet differentiable, thanks to the embedding $H^m \hookrightarrow L^\infty$ for $m \ge 2$ in two dimensions. A natural sufficient condition for convergence is that its Fr\'echet derivative has operator norm strictly smaller than $1$.
For compositions of Fr\'echet-differentiable projections, this reduces to checking transversality of the associated tangent spaces, as shown in \Cref{lemma:comp_transverse}. In our setting this transversality condition is straightforward; this yields a clean convergence result with a contraction-mapping argument on bounded convex domains $\Omega \subset \mathbb{R}^2$, see \Cref{thm:convergenceHm}.

The original splitting algorithm in \cite{dean,caboussat,peruso} is formulated when the projections are taken with respect to the $L^2$ topology for Hessian matrices, see \Cref{sec:conv_l2}. 
In this setting, however, the composition of the projections is not Fr\'echet differentiable, essentially because \(L^2\not\subset L^\infty\), whereas \(L^\infty\) control is needed to ensure that the iterates remain on the elliptic branch. To overcome this difficulty, we develop two different strategies and prove local convergence in \Cref{thm:main,thm:bounded}. For \Cref{thm:main}, the starting point is \Cref{lemma:lipschitz}, which provides an almost-contraction estimate in $H^2$ on the two dimensional torus $\mathbb{T}^2$. We then combine this estimate with \Cref{lemma:frechet-fake}, which shows that the G\^ateaux derivative of the composition of the two projections defines a bounded operator $L^2\to L^2$ and depends continuously on the base point of differentiation in the $L^\infty$ topology. Together, these ingredients allow us to recover a contraction argument and hence linear convergence in $L^2$ of the sequence $D^2u^n$ defined by the splitting method on $\mathbb{T}^2$. In particular, this yields convergence in $H^2$ for $u^n$. As for \Cref{thm:bounded}, we instead consider bounded domains of class $C^4$ and study the spectral radius of the G\^ateaux derivative, computed in \Cref{lemma:frechet-fake},  as a bounded operator from $H^2$ to $H^2$; see \Cref{prop:spectrum}. This requires establishing the regularity of a fourth-order problem arising from the resolvent of the G\^ateaux derivative. The elliptic regularity result stated in \Cref{lemma:pseudo} is proved in Appendix \ref{app:appendix} and is a key ingredient in showing that the spectral radius is strictly smaller than $1$ under the assumption $u\in C^4(\overline{\Omega})$. Consequently, the operator is a contraction with respect to a norm  equivalent to $H^2$, which yields the local convergence of the original splitting algorithm in \Cref{thm:bounded}. While \Cref{thm:bounded} provides a stronger result, the strategy used to prove \Cref{thm:main} may nevertheless be instructive.

We expect that this approach is flexible and may  be adapted to other alternating projections-based schemes for fully nonlinear PDEs, provided analogous regularity and ellipticity assumptions hold.
 
To complement our theoretical results, we present numerical experiments on \(\mathbb{T}^2\), a setting that, to the best of our knowledge, has not previously been considered for this method, using the spatial discretization of \cite{peruso}. Among possible directions for future work is a detailed comparison with Newton's method: while Newton's method may require fewer iterations, each iteration of the splitting scheme is substantially cheaper.

Building on the first version of this work, the first author recently proposed a discrete counterpart of the method developed here, for which convergence is established both with respect to the iteration and in the limit as the mesh size tends to zero; see \cite{peruso2026fully}.

\subsection{Outline}
The remainder of the paper is organized as follows. In \Cref{sec:setting} we recall the method introduced in \cite{dean,caboussat,peruso} and fix notation and assumptions. In \Cref{sec:regular} we analyze the alternating projections viewpoint and prove convergence for the $H^m$-projection variants, $m\ge2$, on bounded convex domains proving \Cref{thm:convergenceHm}. In \Cref{sec:conv_l2} we establish local linear convergence for the original $L^2$-based algorithm in the case $\Omega=\mathbb{T}^2$, while in \Cref{sec:bounded} we address the convergence on general bounded convex domains, proving our main result \Cref{thm:bounded}. Numerical examples on $\mathbb{T}^2$ are reported in \Cref{sec:num}.

\section{Problem setting and algorithm}\label{sec:setting}
\subsection{Notation.}
We collect here the notation used throughout the paper. Let \(d\in\mathbb N\). Matrices in \(\mathbb R^{d\times d}\) are denoted by capital letters, e.g.\ \(A\), while matrix fields on a domain \(\Omega\subset\mathbb R^d\) are denoted by bold capital letters, e.g.\ \(\mathbf A\). We write \(\mathbb S^d\subset\mathbb R^{d\times d}\) for the subspace of symmetric matrices. For \(A,B\in\mathbb R^{d\times d}\) we use the Frobenius product \(A:B := \operatorname{tr}(A^\top B),\) and the associated Frobenius norm \(|A| := \sqrt{A:A}\).

For \(m\in\mathbb N\), the Sobolev space of matrix fields is denoted by \(H^m(\Omega,\mathbb R^{d\times d})\) and is equipped with the standard \(H^m\)-norm induced by the Frobenius product:
\[
\|\mathbf A\|_{m}^2 := \sum_{|\alpha|\le m}\int_\Omega |D^\alpha \mathbf A(x)|^2\,dx,
\]
with the corresponding inner product defined analogously. We adopt the convention \(H^0(\Omega)=L^2(\Omega)\) and use the shorthand \(\|\cdot\| := \|\cdot\|_0\). Given a subset \(M\subset H^m(\Omega,\mathbb R^{d\times d})\), we denote by \(\Pi^{(m)}_M\) the orthogonal projection onto \(M\) with respect to the \(H^m\) inner product (and write \(\Pi := \Pi^{(0)}\)).

Let \(X\) be a Banach space. For a bounded linear operator $F:X\to X$, namely \(F\in\mathcal L(X)\), we denote by \(\|F\|_{\mathcal L(X)}\) its operator norm:
\[
\|F\|_{\mathcal L(X)} := \sup_{\|x\|_X\le 1}\|Fx\|_X.
\]
For a  map \(F:X\to X\), we denote by \(DF(x_0 ;y)\) the G\^ateaux derivative of \(F\) at \(x_0\in X\) in the direction \(y\in X\), whenever the
limit exists:
\[
DF(x_0; y) := \lim_{\varepsilon\to 0}\frac{F(x_0+\varepsilon y)-F(x_0)}{\varepsilon}.
\]
If the Banach space is finite dimensional we denote the G\^ateaux derivative by $\dd F (x_0; y)$.

\subsection{Model problem and algorithm.}
Let \(\Omega \subset \mathbb{R}^2\) be a bounded, convex domain with boundary \(\partial\Omega\). Assume that \(f \in C^0(\overline{\Omega})\) satisfies \(f \ge c_0 > 0\) in \(\Omega\), and let \(g \in H^{3/2}(\partial\Omega)\). The elliptic Dirichlet Monge-Amp\`ere problem is given by
\begin{equation}\label{eq:prob}
\begin{cases}
\det D^2 u = f \quad &\text{in } \Omega,\\
u = g \quad &\text{on } \partial\Omega,
\end{cases}
\end{equation}
where the unknown function \(u\) is required to be convex and \(D^2u\) denotes its Hessian, i.e. \([D^2u]_{ij} = \partial_{x_i x_j} u\). Equation \eqref{eq:prob} is a fully nonlinear elliptic PDE prescribing the product of the eigenvalues of the Hessian of \(u\). This stands in contrast to the classical linear Poisson equation \(-\Delta u = f\), which governs their sum. The convexity constraint on \(u\) is essential: it ensures (degenerate) ellipticity of the Monge-Amp\`ere operator and allows for regularity theory. Under additional regularity assumptions on \(\Omega\) and \(f\), classical solutions \(u \in C^2(\overline{\Omega})\) exist; we refer to \cite{dephilippis} and references therein for sharp results in this direction. The Monge-Amp\`ere equation arises in a variety of applications, including the prescribed Gaussian curvature problem, as well as models in meteorology and fluid mechanics \cite{feng}. Monge-Amp\`ere type equations also arise in optimal transport, particularly in the study of regularity and singularities of transport maps \cite{dephilippis,villani}.

We adopt the nonlinear least-squares framework of \cite{dean,caboussat}, which introduces an auxiliary variable. Define \(\mathbf P := D^2u \in L^2(\Omega,\mathbb S^2)\). Then \eqref{eq:prob} can be rewritten as
\begin{equation}\label{eq:prob1}
\begin{cases}
\text{\rm det}\:\mathbf{P} = f \quad &\text{in } \Omega, \\
\mathbf{P} = D^2u\quad &\text{in } \Omega, \\
u = g \quad &\text{on } \partial\Omega.\\
\end{cases}
\end{equation}
Since we seek a convex solution to \eqref{eq:prob}, we impose the additional requirement that $\mathbf{P}$ is symmetric and positive definite, hereafter spd. We now introduce the functional spaces and sets used to formulate the solution of \eqref{eq:prob1}:
\begin{align} \label{d:V-g}
    \mathcal{V}_g^m &:= \{ D^2v \in H^m(\Omega,\mathbb{S}^2) :\, v \in H^{2+m}(\Omega,\mathbb{R}),\ v|_{\partial\Omega} = g \} \,, \\
    \mathcal{V}_0^m &:= \{ D^2v \in H^m(\Omega,\mathbb{S}^2) :\, v \in H^{2+m}(\Omega,\mathbb{R}),\ v|_{\partial\Omega} = 0 \},
\end{align}
and 
\begin{equation} \label{d:B}
    \mathcal{B}^m := \{ \mathbf{Q} \in H^m(\Omega,\mathbb{S}^2) :\ \det \mathbf{Q}(x) = f(x)\ \text{a.e. in } \Omega,\ \mathbf{Q}(x)\ \text{spd a.e. in } \Omega  \}.
\end{equation}
For the remainder of this section, we set $m = 0$ and write $\mathcal{V}_g := \mathcal{V}_g^0$, $\mathcal{V}_0 := \mathcal{V}_0^0$, and $\mathcal{B} := \mathcal{B}^0$.

It follows that in \eqref{eq:prob1} we seek $\mathbf{P} \in \mathcal{B}$ and $D^2u \in \mathcal{V}_g$. To determine the pair $(D^2u, \mathbf{P})$, we reformulate \eqref{eq:prob1} as the nonlinear least-squares problem
\begin{equation}\label{eq:leastsq}
(D^2u, \mathbf{P}) 
= \argmin_{D^2v \in \mathcal{V}_g,\; \mathbf{Q} \in \mathcal{B}}
\| D^2v - \mathbf{Q} \|^2.
\end{equation}
As observed in \cite{peruso}, we note that \eqref{eq:leastsq} may admit a solution even in cases where \eqref{eq:prob1} does not, namely when $\mathcal{V}_g \cap \mathcal{B} = \emptyset$. However, whenever \eqref{eq:prob1} has a solution, it also satisfies \eqref{eq:leastsq}, and moreover $\|D^2u - \mathbf{P}\| = 0$. To approximate the solution of \eqref{eq:leastsq}, we employ the splitting algorithm proposed in \cite{dean,caboussat}, which iteratively decomposes the minimization problem \eqref{eq:leastsq} into two subproblems. Specifically, given an initial guess $\mathbf{P}^0 \in L^2(\Omega,\mathbb{S}^2)$, for each $n \ge 0$ we seek $D^2u^n$ and $\mathbf{P}^{n+1}$ such that:
\begin{subequations}\label{eq:splitting}
\begin{align}
\label{eq:biharmonic} D^2u^{n} = &\argmin_{D^2v\in \mathcal{V}_g}\|D^2v-\mathbf{P}^n\|^2,\\
\label{eq:firstmin}\mathbf{P}^{n+1} =& \argmin_{\mathbf{Q}\in \mathcal{B}}\|D^2u^n-\mathbf{Q}\|^2.
\end{align}
\end{subequations}
This approach is an instance of alternating minimization, also known as a block coordinate descent (or Gauss-Seidel-type) scheme, and it decouples the nonlinear constraint from the variational part of the problem. More precisely, the nonlinearity is confined to the second subproblem \eqref{eq:firstmin}, while the first subproblem \eqref{eq:biharmonic} retains the variational structure and reduces to a linear biharmonic-type boundary value problem. The second subproblem is a pointwise minimization and can be solved explicitly by a Lagrange multiplier argument; see \cite{glowinski}. The first subproblem leads to a linear biharmonic-type boundary value problem. Conforming finite element discretizations of \eqref{eq:biharmonic} were introduced
in \cite{caboussat} and recently improved in \cite{peruso}. In particular, on polygonal domains \eqref{eq:biharmonic} can be reduced to two uncoupled second-order elliptic problems. Each of these is approximated using continuous linear finite elements, and the Hessian \(D^2u\) is then
recovered by a post-processing technique. The main appeal of the method is that each iteration amounts to solving only low-cost, decoupled
subproblems, while still delivering optimal-order accuracy in practice; for instance, \cite{peruso} observes optimal-order convergence with respect to the mesh size. Moreover, they also report convergence with respect to the iteration.

In this work, we address the open problem of the convergence of the sequence $(D^2u^{n}, \mathbf{P}^n)$ at the continuous level. To this end, we observe that \eqref{eq:splitting} constitutes a particular instance of a block coordinate descent method, namely an alternating projections scheme onto the sets $\mathcal{V}_g$ and $\mathcal{B}$. Indeed, \eqref{eq:splitting} can be equivalently written as 
\begin{equation}\label{eq:projj}
D^2u^{n} = \Pi_{\mathcal{V}_g}(\mathbf{P}^n), 
\quad 
\mathbf{P}^{n+1} = \Pi_{\mathcal{B}}(D^2u^n),
\quad 
\forall n \ge 0, 
\quad 
\mathbf{P}^0 \text{ given},
\end{equation}
where $\Pi_{\mathcal{B}} : L^2 \to \mathcal{B}$ and $\Pi_{\mathcal{V}_g} : L^2 \to \mathcal{V}_g$ denote the projections onto $\mathcal{B}$ and $\mathcal{V}_g$, respectively, defined as in \eqref{eq:splitting}. The operator $\Pi_{\mathcal{V}_g}$ is well defined, since $\mathcal{V}_g$ is an affine closed subspace. The well-posedness of $\Pi_{\mathcal{B}}$, however, \textit{i.e} the existence of the corresponding minimizer, requires additional discussion because of the nonlinear constraint $\det \mathbf{Q} = f$. Following \cite{glowinski}, we note that, as no derivatives are involved, the projection is  well-defined pointwise almost everywhere in $\Omega$. Moreover, the projection is locally smooth, and the minimizer is  unique in a neighbourhood of $\PP$. In order to study the convergence of \eqref{eq:splitting}, we introduce the operator $T: L^2\to\mathcal{B}$, defined as the composition of the two projections, i.e. 
\begin{equation}\label{eq:defT}
T:=\Pi_{\mathcal{B}}\circ \Pi_{\mathcal{V}_g}.    
\end{equation}
Accordingly, the iterative scheme \eqref{eq:projj} can be written as $\mathbf{P}^{n+1} = T(\mathbf{P}^n)$, $n\geq 0$.

While the convergence of the alternating projections method is well understood for finite-dimensional settings \cite{alternating}, the infinite-dimensional case requires more careful analysis. In general, one seeks to show that the composition of the two projections has a norm strictly less than one, so that the Banach fixed-point theorem can be applied. When at least one of the projections is onto a nonlinear set, it is customary to linearize the operator near the solution to obtain a local convergence result. To establish that the norm is strictly less than one, one can employ the notion of transverse spaces, see \Cref{lemma:comp_transverse}.

In the following, we assume the validity of the property:
\begin{enumerate}
    \item[\refstepcounter{assump}\label{A1}(\theassump)]  $\mathbf{P} \in \mathcal{B}$ is uniformly elliptic, i.e., there exist constants $\nu_1, \nu_2 > 0$ such that
    \[
    \nu_1 |\xi|^2 \le \xi^T \mathbf{P}(x) \xi \le \nu_2 |\xi|^2, 
    \quad \forall \,\xi \in \mathbb{R}^2, \ \text{a.e. in } \Omega.
    \]
\end{enumerate}
We further note that, due to the symmetry of $\mathbf{P}$, (A1) implies that $\mathbf{P} \in L^{\infty}(\Omega, \mathbb{S}^2)$. This is the standard assumption ensuring that the Monge-Amp\`ere operator remains in its elliptic regime. However, it may exclude, for example, viscosity solutions, which are generally not covered by the framework of our method.

\section{Convergence of the $H^m$ splitting algorithm}\label{sec:regular}
The splitting iteration \eqref{eq:splitting} is the one proposed in the seminal work \cite{caboussat}. When formulated with $L^2$-projections, the associated map $T: L^2 \to L^2$ is not Fr\'echet differentiable as a map of $L^2(\Omega,\mathbb{S}^d)$, and the convergence analysis therefore needs more delicate arguments; see \Cref{sec:conv_l2}.

If, however, the exact Monge-Amp\`ere solution is sufficiently regular, one can consider the same iteration with projections taken in $H^m(\Omega,\mathbb{S}^d)$ for $m$ large enough so that $H^m\hookrightarrow L^\infty$ (e.g.\ $m> d/2$), i.e.:  
\begin{equation}\label{eq:alg-hm}
D^2u^{n} = \argmin_{D^2v\in \mathcal{V}_g^m}\|D^2v-\mathbf{P}^n\|^2_m,\quad \mathbf{P}^{n+1} = \argmin_{\mathbf{Q}\in \mathcal{B}^m}\|D^2u^n-\mathbf{Q}\|^2_m.
\end{equation}
In this case, the composition map $T^{(m)} :=\Pi^{(m)}_{\mathcal{B}^m}\circ \Pi^{(m)}_{\mathcal{V}^m_g} : H^m \to H^m$ is Fr\'echet differentiable on $H^m$, and the convergence proof reduces to a  fixed-point argument showing that
\[
\|DT^{(m)}\|_{\mathcal{L}(H^m)}<1
\]
there, yielding a contraction. Establishing convergence in this higher-regularity setting serves two purposes. First, it introduces notions and estimates that will be reused in the $L^2$-based analysis. Second, it yields a convergence proof for a version of the method that is not inherently restricted to two dimensions and can be formulated for $d>2$, albeit at higher computational cost. For clarity, we present the proof in the case \(d=2\); but we indicate the (minor) modifications
needed to extend the argument to higher dimensions.

In what follows, we fix $m\in\mathbb{N}$ such that $H^m(\Omega)\hookrightarrow L^\infty(\Omega)$. In particular, for $d=2$ it suffices to take $m\ge 2$. We now prove that $\Pi_{\mathcal{B}^m}^{(m)}$ is Fr\'echet differentiable.
\begin{lemma}\label{lemma:frechetHm}
Let $f\in H^m(\Omega)$. Let $\mathbf{P}\in \mathcal{B}^m\cap \mathcal{V}_g^m$ such that \eqref{A1} is satisfied. Then $\Pi^{(m)}_{\mathcal{B}^m}$ is Fr\'echet differentiable and 
$$D\Pi^{(m)}_{\mathcal{B}^m}(\mathbf{P}) = \Pi^{(m)}_{\ker (\cof\mathbf{P})},$$
where $\ker (\cof\mathbf{P}):=\{\mathbf{X}\in H^m(\Omega,\mathbb{S}^2):\, \cof\mathbf{P}(x):\mathbf{X}(x) = 0\text{ a.e. }x\in\Omega\}$\footnote{Note that this is the kernel of the pointwise operator $\cof (\mathbf{P}):L^2 \to L^2 $ defined as $\cof (\PP) (\XX) (x)= (\cof (\PP) (x) : \XX (x)) \cof (\PP)(x) \in L^2$, where the last holds true since $\PP \in L^\infty$.} and $\Pi^{(m)}_{\ker (\cof\mathbf{P})}:H^m\to H^m$ is the projection on $\ker (\cof\mathbf{P})$. Moreover, $DT^{(m)}(\mathbf{P}) = \Pi_{\ker (\cof\mathbf{P})} \circ \Pi_{\mathcal{V}_0^m}$.
\end{lemma}
\begin{proof}
Define
\[
\Phi: H^m(\Omega,\mathbb{S}^2) \to H^m(\Omega), 
\qquad 
\Phi(\PP)(x) := \det \PP(x) - f(x).
\]
Since $H^m \hookrightarrow L^\infty$ and $\det: \mathbb{S}^2 \to \mathbb{R}$ is smooth, the map $\PP \mapsto \det \PP$ is $C^\infty$ as a map $H^m(\Omega,\mathbb{S}^2) \to H^m(\Omega)$. Hence, $\Phi$ is $C^\infty$. For $\XX \in H^m(\Omega,\mathbb{S}^2)$, the pointwise derivative of the determinant gives
\[
D\Phi(\PP ; \XX) (x) = \dd \det  ( \PP(x) ;\XX(x)) = \cof \PP(x) : \XX(x),
\]
so $D\Phi(\PP; \cdot )$ is a bounded linear operator $H^m(\Omega,\mathbb{S}^2) \to H^m(\Omega)$. Let $g \in H^m(\Omega)$, define
\[
\mathbf{H}_g(x) := \frac{g(x)}{2 \det \PP(x)} \, \PP(x).
\]
Since $\det \PP \ge c_0 > 0$, we have $\mathbf{H}_g \in H^m(\Omega,\mathbb{S}^2)$ and
\[
D\Phi(\PP ;\mathbf{H}_g) (x) = \cof \PP(x) : \mathbf{H}_g(x) = g(x).
\]
Thus $D\Phi(\PP; \cdot )$ is surjective, with bounded linear right inverse $g \mapsto \mathbf{H}_g$. Its kernel is
\[
\ker D\Phi(\PP) = \{ \mathbf{H} \in H^m(\Omega,\mathbb{S}^2) : \cof \PP : \mathbf{H} = 0 \text{ a.e. in } \Omega \} =: T_\PP \mathcal{B}^m.
\]
By the Banach implicit function theorem, $\mathcal{B}^m = \Phi^{-1}(0)$ is a $C^\infty$ submanifold of $H^m(\Omega,\mathbb{S}^2)$ near $\PP$. The differential of a local smooth projection $\Pi_{\mathcal{B}^m}$ onto this submanifold at $\PP$ is the projection onto $T_\PP \mathcal{B}^m$, i.e., $D\Pi_{\mathcal{B}^m}(\PP; \cdot ) = \Pi_{\ker (\cof \PP)} (\cdot )$, see e.g. \cite{abraham,lang}.
Finally, since $T^{(m)} = \Pi^{(m)}_{\mathcal{B}^m} \circ \Pi^{(m)}_{\mathcal{V}_g^m}$ is a composition of Fr\'echet-differentiable maps and $D\Pi^{(m)}_{\mathcal{V}_g^m} = \Pi^{(m)}_{\mathcal{V}_0^m}$, we have
\[
DT^{(m)}(\PP; \cdot ) = \Pi^{(m)}_{\ker (\cof \PP)} \circ \Pi^{(m)}_{\mathcal{V}_0^m} (\cdot ).
\]
\end{proof}
To show that $\|DT^{(m)}\|_{\mathcal{L}(H^m)}<1$, we rely on the following result. We include a proof of this classical result for completeness, for additional background see \cite{twoproj} and the references therein.
\begin{lemma}\label{lemma:comp_transverse}
Let $W,Z$ be closed linear subspaces of a Hilbert space $(V,\|\cdot \|)$ such that $V = W\oplus Z$, then there exists $0\leq c<1$ such that
$$\|\Pi_W \Pi_Z\|_{\mathcal{L}(V)} \leq c.$$
\end{lemma}
\begin{proof} For any $z\in Z$ we have
\[
  \|\Pi_W z\| = \sup_{\substack{w\in W\\ \|w\|=1}} \langle \Pi_W z, w\rangle 
  = \sup_{\substack{w\in W\\ \|w\|=1}} \langle z, w\rangle,
\]
since $\langle \Pi_W z,w\rangle = \langle z,w\rangle$ for $w\in W$.  
Hence
\[
  \|\Pi_W \Pi_Z\|_{\mathcal{L}(V)} = \sup_{\substack{z\in Z\\ \|z\|=1}} \|\Pi_W z\|
  = \sup_{\substack{w\in W,\,z\in Z\\ \|w\|=\|z\|=1}} |\langle w,z\rangle|.
\]
Now assume $V = W \oplus Z$ and suppose, by contradiction, that $\|\Pi_W \Pi_Z\|_{\mathcal{L}(V)}=1$.  
Then there exist at least two sequences $(w_n)\subset W$, $(z_n)\subset Z$ with $\|w_n\|=\|z_n\|=1$ such that $\langle w_n,z_n\rangle \to 1$. Now, compute
\[\|w_n - z_n\|^2 = \|w_n\|^2 + \|z_n\|^2 - 2\langle w_n,z_n\rangle = 2 - 2\langle w_n,z_n\rangle \to 0,\]
so that $w_n - z_n \to 0$ in $V$. Since $V = W \oplus Z$, the map $T: W\times Z \to V$, $T(a,b) = a+b,
$ has a bounded inverse $T^{-1}$, i.e.
there exists $K>0$ such that
\[
  \|w\| + \|z\| \le K \|w+z\| \quad \forall\, w\in W,\, z\in Z.
\]
Applying this to $(w_n,-z_n)$ gives a contradiction
\[
 2= \|w_n\| + \|z_n\| \le K \|w_n - z_n\|\to 0, \qquad n \to \infty\,.
\]
\end{proof}   
Combining \Cref{lemma:frechetHm} and \Cref{lemma:comp_transverse}, we obtain the following theorem, which completes the convergence analysis in $H^m$ for the sequence $\PP^n$. This yields convergence in $H^{m+2}$ for $u^n$.
\begin{theorem}\label{thm:convergenceHm} 
Let $m \geq 2$, $\Omega \subset \mathbb{R}^2$ be a convex bounded smooth domain and  $u \in C^{m+2} (\bar{\Omega})$ be a solution to \eqref{eq:prob} and $\mathbf{P}=D^2u\in\mathcal{B}^m\cap \mathcal{V}_g^m$ be  such that \eqref{A1} holds. Then, there exists $\delta>0$ such that, for any $\PP^0\in H^m(\Omega,\mathbb{S}^2)$ satisfying $\|\PP - \PP^0\|_m <\delta$, the sequence $\{\PP^n\}_{n\geq 0}$ in \eqref{eq:alg-hm} satisfies
\begin{equation}\label{eq:conv}
\norm{\PP^{n} - \mathbf P}_m \le \rho^n\norm{\PP-\PP^0}_m\quad\forall\,n\ge 0,
\end{equation}
with $\rho <1$. In particular, 
$$\|\PP^n - \mathbf{P}\|_m=\|\PP^n - D^2u\|_m\to 0$$
\end{theorem}
\begin{proof} Let $\mathbf{X}\in H^m(\Omega,\mathbb{S}^2)$. By \Cref{lemma:frechetHm} and the definition of Fr\'echet differentiability
\[
T^{(m)}(\mathbf{X}) - T^{(m)}(\mathbf{P}) = DT^{(m)}(\mathbf P;\mathbf{X}-\mathbf{P}) + R(\mathbf{X}-\mathbf{P}),
\]
with the remainder satisfying $\norm{R(\mathbf{X}-\mathbf{P})}_m = o(\norm{\mathbf{X}-\mathbf{P}}_m)$. We now prove that $\mathcal{V}_0^m$ and $\mathrm{ker}(\cof{\mathbf{P}})$ are transverse, i.e. $\mathcal{V}_0^m \oplus \mathrm{ker}(\cof{\mathbf{P}})=H^m(\Omega,\mathbb{S}^2)$. In particular, since $H^m $ is a Hilbert space and $\mathcal{V}_0^m$ and $\mathrm{ker}(\cof{\mathbf{P}})$ are closed, it is sufficient to prove that (a) $\mathcal{V}_0^m + \mathrm{ker}(\cof{\mathbf{P}})=H^m(\Omega,\mathbb{S}^2)$ and (b) $\mathcal{V}_0^m \cap\mathrm{ker}(\cof{\mathbf{P}})=\{\mathbf{0}\}$. To prove (a), it suffices to show that $(\cof\PP)(\mathcal{V}_0^m)=\operatorname{im}(\cof\PP)$. Hence, we need to prove the surjectivity of $\cof(\mathbf{P})$ on $\mathcal{V}_0^m$. This translates into showing that, for any $\phi\in H^m(\Omega)$, there exists a solution to the differential problem
\begin{equation}\label{eq:nondiv}
  \cof{\mathbf{P}}: D^2v = \phi\quad \text{in }\Omega,\quad  v = 0 \quad\text{on }\partial\Omega\,.  
\end{equation}
The existence of a solution $v \in H^1$ is guaranteed  by Lax-Milgram\footnote{Observe that $\diver \cof (\PP) =0$ so the problem is equivalent to the divergence form $\diver (\cof (\PP) \nabla v) = \phi$} and by elliptic regularity results we have that $v \in H^{m+2}$, see for instance \cite[Theorem 15.2]{ADN},
assuming that \eqref{A1} and 
$\cof (\PP) \in C^{m }(\bar{ \Omega})$ hold. Hence $\mathcal{V}_0^m + \mathrm{ker}(\cof{\mathbf{P}})=H^m(\Omega,\mathbb{S}^2)$. On the other hand, let us assume that $D^2w\in \mathcal{V}_0^m\cap \mathrm{ker}(\cof{\mathbf{P}})$. Then, $D^2w$ is such that  
$$\cof{\mathbf{P}}: D^2w = 0\quad \text{in }\Omega,\quad  w = 0\quad \text{on }\partial\Omega\,.$$
Hence, by uniqueness of solutions of the elliptic problem, $ w \equiv {0}$ and this proves that $\mathcal{V}_0^m \cap\mathrm{ker}(\cof{\mathbf{P}})=\{\mathbf{0}\}$. By \Cref{lemma:comp_transverse}, there exists $\rho_0<1$ such that $\|DT^{(m)}(\mathbf{P})\|_{\mathcal{L}(H^m)}=\rho_0$. Now, let $\varepsilon>0$ such that $\rho:=\rho_0+\varepsilon<1$. Furthermore, there exists $\delta$ such that for any $\mathbf{X}\in H^m(\Omega,\mathbb{S}^2)$ satisfying $\|\XX - \PP\|_m\leq \delta$, $\norm{R(\mathbf{X}-\mathbf{P})}_m \leq \varepsilon\norm{\mathbf{X}-\mathbf{P}}_m$. Then 
$$\norm{T^{(m)}(\XX) - \mathbf P}_m \le \rho\,\norm{\XX - \mathbf P}_m\quad\forall\,n\ge 0.$$
This proves the linear convergence.
\end{proof}
\begin{remark}\label{rem:transversality}
In the proof of \Cref{thm:convergenceHm}, the result that $\ker(\cof \mathbf{P})$ and $\mathcal{V}_0^m$ are transverse holds true also for $m=0$. In particular, when $d=2$ and $m=0$, assumption \eqref{A1} alone implies
\[
\|\Pi_{\ker(\cof \mathbf{P})}\circ \Pi_{\mathcal{V}_0}\|<1.
\]
Indeed, this follows directly from the Miranda-Talenti and Cordes inequalities, since $\mathbf{P}$ is uniformly elliptic by assumption.
\end{remark}
\begin{remark}\label{rem:cordes}
For \(d>2\), assumption \eqref{A1} is no longer sufficient to guarantee the existence of a solution \(\phi\in H^m(\Omega)\) to \eqref{eq:nondiv}. Stronger structural hypotheses are required; for instance, one may impose the Cordes condition \cite{ADN}.
\end{remark}

We have shown that, when the projections are taken in $H^m$ with $m > \frac{d}{2}$ so that $H^m(\Omega)\hookrightarrow L^\infty(\Omega)$, the associated map $T$ is a local contraction and the iteration converges linearly. Although this higher-regularity variant is primarily of theoretical interest (and can be formulated in any space dimension), it isolates the key mechanism behind the method and provides a useful template for the $L^2$-based analysis developed next.

\section{Convergence of the $L^2$ splitting algorithm on $\mathbb{T}^2$}\label{sec:conv_l2}
The fixed-point argument used in \Cref{sec:regular} relies on working in a Sobolev space $H^m(\Omega)$ which is a Hilbert space with $H^m(\Omega)\hookrightarrow L^\infty(\Omega)$, which ensures that the relevant projection map is Fr\'echet differentiable and that the iteration can be shown to be a local contraction. This approach breaks down for the original 
$L^2$-projection formulation: without 
an $L^\infty$ control, the sequence of matrices may leave the set of positive definite matrices, and the composition of the projections may fail to be Fr\'echet differentiable.

In this section we nevertheless prove a local linear convergence result for the $L^2$-based iteration on $\Omega=\mathbb{T}^2$. Assuming the Monge-Amp\`ere solution is sufficiently regular (in particular, $u\in H^4(\mathbb{T}^2)$) and that the initial guess is sufficiently close to $u$, we show that the iterates remain under control in $H^2$; see \Cref{lemma:lipschitz}. By interpolation, this yields convergence in $H^{3/2}$ for the sequence $\PP^n$, and hence in $L^2$. In particular, the sequence $\PP^n$ stays in the set of uniformly positive definite matrices since $H^{3/2} \hookrightarrow L^\infty$.

As is standard in the periodic setting, we write the convex potential in the form
\(\tfrac{|x|^2}{2}+u(x)\) with \(u\) periodic and mean zero. With a slight abuse of notation, we thus consider the problem on \(\mathbb T^2\): find \(u\) such that
\begin{equation}\label{eq:torus}
    \det (\mathbf{I}+D^2u) = f\quad \text{on }\mathbb{T}^2,\quad \int_{\mathbb{T}^2} u= 0, \quad \frac{|x|^2}{2}+ u \text{ convex}.
\end{equation}
Here \(\mathbf{I}\) denotes the identity matrix in \(\mathbb{R}^{2\times 2}\). The splitting algorithm \eqref{eq:splitting} is unchanged; however, for the periodic setting we redefine the constraint sets \(\mathcal V\) and \(\mathcal B\) so as to be consistent with \eqref{eq:torus}. We set
\begin{equation*} \label{eq:V_new}
    \mathcal{V} := \{ D^2v \in L^2(\mathbb{T}^2,\mathbb{S}^2) :\, v \in H^{2}(\mathbb{T}^2,\mathbb{R}),\ \int_{\mathbb{T}^2}v= 0 \} \,,
\end{equation*}
and
\begin{equation*}
    \mathcal{B} := \{ \mathbf{Q} \in L^2(\mathbb{T}^2,\mathbb{S}^2) :\ \det((\mathbf{I} +  \mathbf{Q})(x)) = f(x)\ \text{a.e. in } \mathbb{T}^2,\ (\mathbf{I}+\mathbf{Q})(x)\ \text{spd a.e. in } \mathbb{T}^2  \}.
\end{equation*}
The splitting algorithm on the torus $\mathbb{T}^2$, similarly to \Cref{sec:setting}, is defined as the iterative scheme 
\begin{equation} \label{eq:algorithm-torus}
D^2u^{n} = \argmin_{D^2v\in \mathcal{V}}\|D^2v-\mathbf{P}^n\|^2,\quad \mathbf{P}^{n+1} = \argmin_{\mathbf{Q}\in \mathcal{B}}\|D^2u^n-\mathbf{Q}\|^2 \,.
\end{equation}
Therefore, the iteration map $T:L^2(\mathbb{T}^2, \mathbb{S}^2)\to \mathcal{B}$ is defined by $T = \Pi_\mathcal{B}\circ \Pi_{\mathcal{V}}$, where $\Pi_{\mathcal{V}}$ is the $L^2$ projection onto $\mathcal{V}$ and $\Pi_{\mathcal{B}}$ is the $L^2$ projection onto $\mathcal{B}$, which is equivalently defined pointwise.
\subsection{Differentiability of $\Pi_{\mathcal{B}}$ and $T$}\label{secc:lipschitz}
In $L^2$, the projection on $\mathcal{B}$ boils down to a pointwise projection, given that no derivatives are involved. For any $x \in \mathbb{T}^2$ we define 
\begin{align*} 
    \mathcal{B}_x:=\{Q\in  \mathbb{S}^2:\det(I +  Q) = f(x),\ I+Q \text{ spd}\},
\end{align*}
and $\Pi_{\mathcal{B}_x}:\mathbb{S}^2\to \mathcal{B}_x$ as the projection on $\mathcal{B}_x$ with respect to the Frobenius norm.  For $\XX\in L^2(\mathbb{T}^2,\mathbb{S}^2)$, it holds: 
\begin{align*}
    \Pi_{\mathcal{B}_x}(\XX(x)) = \Pi_{\mathcal{B}}(\XX)(x)\quad \text{a.e. in }\mathbb{T}^2,
\end{align*}
and 
\begin{align}\label{eq:pointwise-T}
     T(\XX(x))  = (\Pi_{\mathcal{B}} \circ \Pi_{\mathcal{V}} (\XX)) (x)= \Pi_{\mathcal{B}_x} (  \Pi_{\mathcal{V}} (\XX) (x))\quad \text{a.e. in }\mathbb{T}^2.
\end{align}
Moreover, $\mathcal{B}_x$ is a smooth embedded hypersurface of $\mathbb S^2$ and it is known that the projection $\Pi_{\mathcal{B}_x}$ is $C^\infty$ on its maximal open domain \cite{krantz,dudek,Leobacher}. We now define $\dd_M \Pi_{\mathcal{B}_x}$ as the derivative with respect to  $M$ of the map $\Pi_{\mathcal{B}_x} (M)$. We also denote by $\dd_f \Pi_{\mathcal{B}_x}$ as the derivative with respect to $f(x)$ of the map $\Pi_{\mathcal{B}_x}$ where the dependence on $f(x)$ is in the definition of $\mathcal{B}_x$. 
\begin{lemma} \label{lemma:uniform-cont}
Let $f \in C^0 (\mathbb{T}^2)$ and $\XX$ be such that  $\XX (x) \in\mathcal{B}_x$ and  $|(I + \XX (x))^{-1}| \leq M$ for all $x \in \mathbb{T}^2$. Then, there exist $\delta = \delta (M)>0$ and $C>0$ such that
for all $\|\YY - \XX\|_{L^\infty}< \delta $
\begin{align}\label{eq:boundedness-DM}
    \sup_{x \in \mathbb{T}^2}   | \dd_f^i \dd_M^j  \Pi_{\mathcal{B}_x} (\YY (x))| \leq C \,, \quad \forall j=1,2,3\,, i=0,1,2\,.
\end{align} 

Furthermore, 
for all $\eps>0$ there exists $\delta = \delta (\eps , M)>0$, such that for all $\|\YY - \XX\|_{L^\infty}< \delta $
\begin{align} \label{eq:continuity-DM}
    \sup_{x \in \mathbb{T}^2} | \dd_M \Pi_{\mathcal{B}_x}(\YY (x)) - \dd_M \Pi_{\mathcal{B}_x}(\XX(x)) | < \eps \,.
\end{align} 
\end{lemma}
\begin{proof}
Let $Y$ be such that $I + Y$ is spd and elliptic, then in \cite{Leobacher} it is proved that 
\begin{align} \label{eq:explicit-DM}
    \dd_M\Pi_{\mathcal{B}_x}(Y)=\big(I-d\,L_{\Pi_{\mathcal{B}_x}(Y)}\big)^{-1}\Pi_{\ker(\cof(I+  \Pi_{\mathcal{B}_x}(Y)))}\,,
\end{align}
where $d:=|Y-\Pi_{\mathcal{B}_x}(Y)|$, $\ker(\cof(I+  \Pi_{\mathcal{B}_x}(Y)))$ is the tangent space to $\mathcal{B}_x$ at $\Pi_{\mathcal{B}_x}(Y)$ and $L_{\Pi_{\mathcal{B}_x}(Y)}: \ker(\cof(I +  \Pi_{\mathcal{B}_x}(Y)))\to \ker(\cof (I+\Pi_{\mathcal{B}_x}(Y)))$ is the shape operator defined as: 
$$L_{Q}(H) = \frac{1}{|(I+Q)^{-1}|}\Pi_{\ker(\cof (I+Q))}((I+Q)^{-1} H (I+Q)^{-1}) \,.$$
Using that $d \leq 2 \delta$ by $X \in \mathcal{B}_x$, if $ 2 \delta \|L_{\Pi_{\mathcal{B}_x}(Y)} \|_{\mathcal{L}} < 1$ we have 
\begin{align*}
\operatorname{Lip} (\Pi_{\mathcal{B}_x})_{| B(X,\delta)}
& \leq \sup_{Y\in B(X,\delta)} \|\dd_M\Pi_{\mathcal{B}_x}(Y) \|_{\mathcal{L}} \leq \sup_{Y\in B(X,\delta)} \| \big(I-d\,L_{\Pi_{\mathcal{B}_x}(Y)}\big)^{-1}\|_{\mathcal{L}} 
\\
& \leq \sup_{Y\in B(X,\delta)} 1 + \sum_{j \geq 1} (2\delta \| L_{\Pi_{\mathcal{B}_x}(Y)}\|_{\mathcal{L}})^j   \,.
\end{align*}
Finally, using that 
\begin{align*}
\|L_{\Pi_{\mathcal{B}_x}(Y)}\|_{\mathcal{L}} &=
    \sup_{|H| \leq 1} |L_{\Pi_{\mathcal{B}_x}(Y)} (H)|
      \leq  |(I+\Pi_{\mathcal{B}_x}(Y))^{-1}| = |(I+X + \Pi_{\mathcal{B}_x}(Y) -X)^{-1}|
    \\
    & = | (I + (I+X)^{-1}( \Pi_{\mathcal{B}_x}(Y) -X) )^{-1}(I+X)^{-1}|
    \\
    & \leq |(I+X)^{-1}| |(I + (I+X)^{-1}( \Pi_{\mathcal{B}_x}(Y) -X) )^{-1}| \leq M \frac{1}{1- 2 \delta M} \,,
\end{align*}
where in the last inequality we have used Neumann series expansion and $|\Pi_{\mathcal{B}_x}(Y) - X|\leq 2 \delta$. Therefore, if $\frac{\delta M^2}{1- 2 \delta M}< 1 $, we deduce  $ 2 \delta \|L_{\Pi_{\mathcal{B}_x}(Y)} \|_{\mathcal{L}} < 1$ and for $\delta>0$ small enough we deduce \eqref{eq:boundedness-DM} with $j=1$ and $i=0$.
We have just shown that the map
$$ (x, Y) \in \mathbb{T}^2 \times \{ Y \in \mathbb{S}^2: |Y|+ |(I +Y)^{-1}| \leq M\} \to \dd_M\Pi_{\mathcal{B}_x}(Y) $$
is bounded.  Since this map depends smoothly on $Y$ and $x$ by the explicit formula \eqref{eq:explicit-DM} and the fact that the manifold $ \mathcal{B}_x$ depends smoothly on $f$ and $f \in C^0$, we conclude that the map and each derivative $\dd_M^j$ with $j \geq 0$ and $\dd_f^i$ for $i \geq 0$ are continuous on this space from which the result \eqref{eq:boundedness-DM} follows. Finally, using that  $
\mathbb{T}^2 \times \{ Y \in \mathbb{S}^2: |Y|+ |(I +Y)^{-1}| \leq M\}$ is a compact set we deduce uniform continuity, from which \eqref{eq:continuity-DM} follows.
\end{proof}

As previously mentioned, the operator $T:L^2\to L^2$ is not Fr\'echet differentiable because $\Pi_{\mathcal{B}}$ is not and \Cref{lemma:frechetHm} is not valid here. Nevertheless, close to the solution, it is possible to linearize $T$ using the G\^ateaux derivative.
\begin{lemma} \label{lemma:frechet-fake}
  Let $\PP \in \mathcal{B}\cap\mathcal{V}$ be such that $\mathbf{I}+\PP$ satisfies \eqref{A1}, $p\in[1,+\infty]$ and $\delta < \nu_1/2$.  Then the  map
    $$
    D\Pi_{\mathcal B} : \left\{ \mathbf{A}\in L^\infty(\mathbb{T}^2;\mathbb{S}^2) : \|\mathbf{A}-\PP\|_{L^\infty}<\delta \right\} \rightarrow \mathcal{L}(L^p,L^p),
    \qquad
    \mathbf{A} \mapsto D\Pi_{\mathcal B}(\mathbf{A};\,\cdot)
    $$
    where, for every $ \mathbf{B} \in L^p (\mathbb{T}^2; \mathbb{S}^2)$,
    \begin{equation*}
         D \Pi_{\mathcal{B}} (\mathbf{A}; \mathbf{B}) (x): = \dd_M \Pi_{\mathcal{B}_x} (\mathbf{A} (x);  \mathbf{B}  (x)), 
    \end{equation*}
    is continuous  and satisfies, for all $j=1,2,$ and $ i=0,1,2$
    \begin{align} \label{eq:uniformLp}
          \|\dd_M^j \dd^i_f D \Pi_{\mathcal{B}} (\mathbf{A}; \mathbf{B})\|_{L^p } \leq C \| \mathbf{B} \|_{L^p} \,, \quad \forall \mathbf{A} : \| \mathbf{A} - \PP \|_{L^\infty}< \delta \,.
    \end{align}
    Furthermore,
    \begin{equation}\label{eq:linear-T}
       T(\XX) - T (\PP) = \int_0^1 D \Pi_{\mathcal{B}}(\PP  + t  \Pi_{\mathcal{V}} (\XX - \PP) ; \Pi_{\mathcal{V}} (\XX - \PP) )  dt\,, \quad \forall \XX : \,\|\Pi_{\mathcal{V}} (\XX - \PP) \|_{L^\infty} < \delta \,. 
    \end{equation}
\end{lemma} 
\begin{proof}
For any $x \in \mathbb{T}^2 $ fixed, by smoothness of $\Pi_{\mathcal{B}_x}$ with respect to the matrix variable, we have 
 $$ \Pi_{\mathcal{B}_x} (\XX (x)) - \Pi_{\mathcal{B}_x} (\YY (x)) = \int_0^1 \dd_M \Pi_{\mathcal{B}_x} (\YY (x) + t (\XX (x)  - \YY (x)) ;\XX (x) - \YY (x)) dt.$$
The continuity in $\mathbf{A} \in L^\infty$ of the map $D \Pi_{\mathcal{B}} (\mathbf{A} ; \mathbf{B}) $ follows from  the definition and \Cref{lemma:uniform-cont}
 \begin{align*}
     \sup_{\| \mathbf{B} \|_{L^p} \leq 1}\| D \Pi_{\mathcal{B}} (\mathbf{A} ; \mathbf{B}) - D \Pi_{\mathcal{B}} (\mathbf{A} ' ; \mathbf{B}) \|_{L^p} &= \sup_{\| \mathbf{B} \|_{L^p} \leq 1} \| (\dd_M \Pi_{\mathcal{B}_{x}} (\mathbf{A} (\cdot )) - \dd_M \Pi_{\mathcal{B}_x} (\mathbf{A}' (\cdot )) )   \mathbf{B}  (\cdot)\|_{L^p} 
     \\
        & \leq  \|  \dd_M \Pi_{\mathcal{B}_{x}} (\mathbf{A}(\cdot )) - \dd_M \Pi_{\mathcal{B}_x} (\mathbf{A}' (\cdot ))  \|_{L^\infty}  \to 0 
 \end{align*}
 as $\| \mathbf{A}' - \mathbf{A} \|_{L^\infty} \to 0$. The bound \eqref{eq:uniformLp} follows directly from  \Cref{lemma:uniform-cont}.
 
 Finally, by using \eqref{eq:pointwise-T} and observing that $\Pi_{\mathcal{V}} (\PP) = \PP$, $ \Pi_{\mathcal{V}} (\XX) - \PP =  \Pi_{\mathcal{V}} (\XX - \PP)$, we deduce that, for a.e. $x \in \mathbb{T}^2$,
 $$ T(\XX) (x) - T (\PP)(x) = \int_0^1 \dd_M  \Pi_{\mathcal{B}_x}(  \PP (x)  + t   \Pi_{\mathcal{V}} (\XX - \PP) (x);  \Pi_{\mathcal{V}} (\XX - \PP) (x))   dt\,,$$
which yields \eqref{eq:linear-T}.
\end{proof}

We state and prove the final analytical tool required for the convergence proof. In particular, the proof exploits structure of the domain $\mathbb{T}^2$, which is essential to obtain the $(1-\eps)$-Lipschitz continuity of $T$ in $H^1$ and $H^2$ up to a fixed error. In the following, we denote by $|\cdot |_{m}$ the seminorm in $H^m$.
\begin{lemma}\label{lemma:lipschitz}
    Let $\PP \in \mathcal{B}\cap \mathcal{V}\cap H^2(\mathbb{T}^2,\mathbb{S}^2)$ satisfy \eqref{A1}. Then, there exist $\eps, \delta, C>0$  such that  $T$ satisfies 
    \begin{align} \label{eq:lip-H1}
        | T(\XX) - T(\PP) |_{1} \leq (1 - \eps) | \XX - \PP |_{1} + C \,, \qquad \forall \XX \in H^2(\mathbb{T}^2,\mathbb{S}^2) \, : \| \XX - \PP \|_{H^{3/2} } < \delta \,,
    \end{align}
    and 
    \begin{align} \label{eq:lip-H2}
        | T(\XX) - T(\PP) |_{2} \leq (1 - \eps) | \XX - \PP |_{2} + C \,, \qquad \forall \XX \in H^2(\mathbb{T}^2,\mathbb{S}^2) \, : \| \XX - \PP \|_{H^{3/2} } < \delta \,.
    \end{align}
\end{lemma}

\begin{proof}

With the same notation used above, we denote by $\partial_i $ the $i$-th derivative with respect to $x_i$, $\dd_M $ the derivative of $\Pi_{\mathcal{B}_x} (M)$ with respect to the matrix $M$ and $\dd_f $ the derivative of $\Pi_{\mathcal{B}_x}$ with respect to $f$. For any $\XX$ satisfying $\|\XX-\PP\|_{L^\infty}<\delta$, applying the chain rule and \Cref{lemma:frechet-fake} yields
    \begin{align} \label{eq:formula-x-i}
         \partial_i D\Pi_{\mathcal{B}} (\XX ; \YY) =D \Pi_{\mathcal{B}} (\XX ; \partial_i \YY)   + \dd_M D \Pi_{\mathcal{B}} (\XX ; \YY) \partial_i \XX + \dd_f D \Pi_{\mathcal{B}} (\XX; \YY) \partial_i f.
    \end{align}
    Using the regularity of $\Pi_{\mathcal{B}_x}$ by \Cref{lemma:frechet-fake}
    $$ \| \dd_M D\Pi_{\mathcal{B}} (\XX ; \YY) \|_{L^\infty}  \leq C \| \YY \|_{L^\infty} \qquad \| \dd_f D\Pi_{\mathcal{B}} (\XX; \YY) \| \leq C \| \YY \| $$
uniformly for all $\XX$ such that $\|\XX-\PP\|_{L^\infty}<\delta$. Since $f\in C^2$, this implies
\begin{equation}\label{eq:first-derivative-DPi}
\|\partial_i D\Pi_{\mathcal{B}}(\XX;\YY)\|
\le
\|D\Pi_{\mathcal{B}}(\XX;\partial_i\YY)\|
+ C\|\YY\|_{L^\infty}\|\partial_i\XX\|
+ C \| \YY \|.
\end{equation}
Differentiating \eqref{eq:formula-x-i} once more with respect to $x_j$ and applying the chain rule again, we obtain an expression for $\partial_j\partial_i D\Pi_{\mathcal{B}}(\XX;\YY)$ consisting of terms involving derivatives of $\YY$, $\XX$, and $f$ up to second order, as well as products of first derivatives. Using the bounds on the higher-order derivatives of \Cref{lemma:frechet-fake}, the assumption $\|f\|_{C^2}\lesssim 1$, H\"older's inequality, and the condition $\|\XX\|_{W^{1,4}}\le C$ since $H^{3/2} \hookrightarrow W^{1,4}$ and $\| \XX - \PP \|_{H^{3/2}} < \delta$, we deduce
\begin{equation}\label{eq:second-derivative-DPi}
\|\partial_i\partial_j D\Pi_{\mathcal{B}}(\XX;\YY)\|
\le
\|D\Pi_{\mathcal{B}}(\XX;\partial_i\partial_j\YY)\|
+ C\|\YY\|_{L^\infty}\|\partial_i\partial_j\XX\|
+ C\|\YY\|_{W^{1,4}}\,.
\end{equation}
To prove \eqref{eq:lip-H1} and \eqref{eq:lip-H2} we exploit \eqref{eq:linear-T}. Differentiating it with respect to $x_i$ yields
\begin{equation}\label{eq:chain-rule-T}
\begin{aligned}
\partial_i  T(\XX ) - \partial_i T(\PP)
&= \int_0^1  D \Pi_{\mathcal{B}}
\big(\PP  + t  \Pi_{\mathcal{V}} (\XX - \PP) ;
\partial_i \Pi_{\mathcal{V}} (\XX -\PP) \big) \, dt
\\
&\quad + \int_0^1  \dd_M D \Pi_{\mathcal{B}}
\big(\PP  + t  \Pi_{\mathcal{V}} (\XX - \PP) ;
\Pi_{\mathcal{V}} (\XX -\PP) \big)
\, \partial_i ( \PP  + t  \Pi_{\mathcal{V}} (\XX - \PP)) \, dt
\\
&\quad + \int_0^1  \dd_f D \Pi_{\mathcal{B}}
\big(\PP  + t  \Pi_{\mathcal{V}} (\XX - \PP) ;
\Pi_{\mathcal{V}} (\XX -\PP) \big)
\, \partial_i f \, dt .
\end{aligned}
\end{equation}
We observe that \( D\Pi_{\mathcal{B}}(\PP;\Pi_{\mathcal V}(\cdot))=\Pi_{\ker(\cof(I+\PP))}\circ\Pi_{\mathcal V},\)
and therefore, by \Cref{rem:transversality},
\[
\big\|D\Pi_{\mathcal{B}}(\PP;\Pi_{\mathcal V}(\cdot))\big\|_{\mathcal{L}(L^2)}<1.
\]
By continuity of $D\Pi_{\mathcal{B}}$ at $\PP$, see \Cref{lemma:frechet-fake}, there exist $\varepsilon>0$ and $\delta>0$ such that
\begin{equation}\label{eq:strict-contraction}
\big\|D\Pi_{\mathcal{B}}(\YY;\Pi_{\mathcal V}(\cdot))\big\|_{\mathcal{L}(L^2)}
\le 1-2\varepsilon,
\qquad
\forall\,\|\YY-\PP\|_{L^\infty}<\delta.
\end{equation}
Moreover, $\|\Pi_{\mathcal{V}}\|_{\mathcal{L}(L^2)}=1$ and, on
$\mathbb{T}^2$, it also holds true that (see \Cref{remark:projection})
\begin{align} \label{eq:lipsc-1-proj}
    \|\Pi_{\mathcal{V}}\|_{\mathcal{L}(H^m)}=1, \quad m\geq 0.
\end{align}
Using these properties, \eqref{eq:first-derivative-DPi} and \eqref{eq:uniformLp}, we deduce that
    \begin{align*}
        \| \partial_i T (\XX) &- \partial_i T (\PP) \|  = \| \partial_i \int_0^1   D\Pi_{\mathcal{B}} (\PP + t  \Pi_{ \mathcal{V}} (\XX - \PP); \Pi_{ \mathcal{V}} (\XX - \PP)) dt \| 
        \\
        & \leq (1- 2 \eps ) \| \partial_i \Pi_{ \mathcal{V}} ( \XX -  \PP) \| + C \|\Pi_{ \mathcal{V}} ( \XX - \PP ) \|_{L^\infty} (\| \partial_i \Pi_{ \mathcal{V}} (\XX-\PP)\|  + \| \partial_i  \PP \| ) + C \|\Pi_{ \mathcal{V}} ( \XX - \PP ) \|
        \\
        & \leq (1 -2 \eps ) \| \partial_i \XX - \partial_i \PP \| + C \| \XX - \PP \|_{H^{3/2}} (\| \partial_i \XX - \partial_i \PP\|  + 2 \| \partial_i \PP \| ) + C \|\Pi_{ \mathcal{V}} ( \XX - \PP ) \|
        \\
        & \leq (1- \eps)\| \partial_i \XX - \partial_i \PP \| +2 C\,,
    \end{align*}
    where in the last we have used  $\| \XX - \PP \|_{H^{3/2}} < \delta$, $\| \XX \|_{H^{1}} \leq \| \PP \|_{H^1} + \delta \leq C$ and $C \delta < \eps$. Therefore, summing over $i$, we obtain \eqref{eq:lip-H1}. 

Differentiating \eqref{eq:linear-T} twice and using
\eqref{eq:second-derivative-DPi}, \eqref{eq:lipsc-1-proj}, together with the embedding
$H^{3/2}\hookrightarrow W^{1,4}$ and $H^{3/2}\hookrightarrow L^\infty$, we obtain
\begin{align*}
\|\partial_i\partial_j T(\XX)-\partial_i\partial_j T(\PP)\|
&\le (1-2\varepsilon)\|\partial_i\partial_j(\XX-\PP)\|
\\
&\quad
+ C\|\XX-\PP\|_{H^{3/2}}
\big(\|\partial_i\partial_j(\XX-\PP)\|+\|\partial_i\partial_j\PP\|\big)
+ C\|\XX-\PP\|_{H^{3/2}}\,.
\end{align*}
Choosing $\delta$ such that $C\delta<\varepsilon$ and summing over $i,j$
yields \eqref{eq:lip-H2} up to not relabeling the constant $C$.

\end{proof}
\begin{remark} \label{remark:projection}
If $\Omega\neq \mathbb{T}^2$ is a bounded convex domain, one can still obtain a stability estimate but would lose the contraction property in $H^m$ for $m>0$: 
\begin{align*}
    | T(\XX) - T(\PP) |_{m} \leq (1 - \eps)C_{m,\Omega}| \XX - \PP |_{m} + C \,, \qquad \forall \XX \in H^2(\Omega,\mathbb{S}^2) \, : \| \XX - \PP \|_{H^{3/2} } < \delta \,.
\end{align*}
where $C_{m,\Omega}= \sup_{\| \PP \|_m \leq 1}\|\Pi_{\mathcal{V}} (\PP)\|_{m}$ depends on the elliptic regularity constant and on the trace operator.  Indeed on the torus the projection $D^2 u = \Pi_{\mathcal{V}} (\PP)$ is the zero average solution to 
$$ \int D^2 u : D^2 v = \int \PP : D^2 v \qquad \forall v \in H^2$$ for which it follows by Cauchy-Schwarz $\| D^2 u \| \leq \| \PP \|$ and the choice $v =u$. We observe that $u $ also solves 
$$ \int D^\alpha D^2 u : D^2 v = \int D^\alpha  \PP : D^2 v \qquad \forall v \in H^2$$ 
for any multi-index $\alpha$. Then 
one could take $v =D^\alpha u$ to conclude also $\| D^\alpha D^2u \| \leq \| D^\alpha \PP \|$ from which $\|u \|_m \leq \| \PP \|_m$ for any $m \geq 0$. On a bounded domain $\Omega \subset \mathbb{R}^2$, the choice $v = D^\alpha u$ is in general not admissible because of boundary conditions. In particular, the resulting inequality contains boundary terms that require control via trace inequalities and elliptic regularity theory.

\end{remark}

\subsection{Local convergence of \eqref{eq:algorithm-torus} on $\mathbb{T}^2$}
We now establish local linear convergence of the splitting method \eqref{eq:algorithm-torus} on $\mathbb{T}^2$. Assuming that the exact solution is sufficiently regular and uniformly elliptic, we prove that if the initial guess $\PP^0$ is close enough in $H^2$ to the Hessian matrix  $\PP =D^2u $ of the solution to the Monge-Amp\`ere equation, then the iteration of the splitting method \eqref{eq:algorithm-torus} contracts in $L^2$ and hence the sequence $D^2 u^n$ defined in \eqref{eq:algorithm-torus} converges in $L^2$ at a linear rate. This yields convergence in $H^2$ for $u^n$. Furthermore, the convergence of the sequence $D^2 u^n$ also holds true in $H^{3/2}$.

\begin{theorem}\label{thm:main}
Assume $f\in C^2(\mathbb{T}^2)$ and let $u\in H^4(\mathbb{T}^2)$ be a solution of \eqref{eq:torus}. Set $\mathbf{P}:=D^2u$ and assume that $\mathbf{I}+\mathbf{P}$ satisfies \eqref{A1}. Then there exist constants $\delta>0$ and $\rho \in(0,1)$ such that, for any initial guess $\mathbf{P}^0\in H^2(\mathbb{T}^2,\mathbb{S}^2)$ with $\|\mathbf{P}^0-\mathbf{P}\|_{2}<\delta$, the sequence $(\mathbf{P}^n)_{n\ge0}$ generated by \eqref{eq:splitting} satisfies
\[
\|\mathbf{P}^n-\mathbf{P}\|\le  \rho^n \|\mathbf{P}^0-\mathbf{P}\|\qquad\text{for all }n\ge0.
\]
Hence, $\mathbf{P}^n\to \mathbf{P}$ in $L^2(\mathbb{T}^2,\mathbb{S}^2)$ at a linear rate. Moreover, $\PP^n\to \PP$ also in $H^s(\mathbb{T}^2, \mathbb{S}^2)$, $s\leq \frac{3}{2}$.
\end{theorem}
\begin{proof} 
Firstly, we recall that, by  \Cref{rem:transversality} and the definition of $D \Pi_{\mathcal{B}}$ in \Cref{lemma:frechet-fake} and \eqref{eq:explicit-DM}, we have 
\begin{align}\label{eq:constant<1}
    \|D \Pi_{\mathcal{B}}(\PP ;  \Pi_{\mathcal{V}}( \cdot))\|_{\mathcal{L}(L^2)}= \| \Pi_{\ker(\cof(  \mathbf{I}+\PP))} \circ \Pi_{\mathcal{V}} \|_{\mathcal{L}(L^2)}= \rho_0 <1.
\end{align}
Denoting by $C>0$ the constant of the Sobolev embedding $H^{3/2} \hookrightarrow L^\infty$ meaning $\| \cdot \|_{L^\infty} \leq C \| \cdot \|_{H^{3/2}}$, we observe that 
$$ K= \{ \XX \in L^2(\mathbb{T}^2,\mathbb{S}^2) : \| \XX \|_{H^{3/2}} \leq  \delta^{1/8} \}$$
is a compact space in $L^2(\mathbb{T}^2,\mathbb{S}^2)$ with respect to the $L^2$ topology
and for $\delta>0$ sufficiently small by \eqref{eq:constant<1} and \Cref{lemma:frechet-fake} that says that $\YY \in L^\infty \to D \Pi_{\mathcal{B}} (\YY; \cdot) \in \mathcal{L} (L^2,L^2) $ is continuous, we deduce  
\begin{align*}
\| T(\XX) - T (\PP)\|  & \leq \int_0^1 \| D \Pi_{\mathcal{B}}(\PP  + t  \Pi_{\mathcal{V}} (\XX - \PP) ; \Pi_{\mathcal{V}} (\XX - \PP) ) \|  dt
\\
& \leq (\rho_0 + \eps) \| \Pi_{\mathcal{V}} (\XX - \PP)\| \leq (\rho_0 + \eps  ) \|\XX - \PP \| 
\end{align*}
where we have used also  $\| \Pi_{\mathcal{V}} \|_{\mathcal{L}(L^2)}\leq 1$.
In particular, if we choose $\delta >0$ small enough in the definition of $K$ depending on $\rho < 1$,  we can impose that ${\rho} = \rho_0 + \eps <1 $.

We now prove that for $\delta >0$ sufficiently small $\PP^n - \PP \in K$ for all $n \in \NN$. Indeed, for $n=0$, it is given by the assumption. By induction, we  assume it to be true for any $k \leq n-1$ and we claim it for $k=n$. First we observe that
\begin{align*}
    \| \PP^n -  \PP \| = \| T (\PP^{n-1}) - T (\PP) \| = \| T (\PP + (\PP^{n-1} - \PP)) - T (\PP) \| \leq  \rho \| \PP^{n-1} - \PP \| ,
\end{align*}
hence, 
\begin{align} \label{eq:proof-conv-L2}
    \| \PP^n -  \PP \| \leq \rho^n \| \PP^0 - \PP \| \,.
\end{align}
Furthermore, by \eqref{eq:lip-H1} and \eqref{eq:lip-H2} in \Cref{lemma:lipschitz}, if $\delta >0$ is sufficiently small,  we have
\begin{align*}
    \| \PP^n -  \PP \|_{2} = \| T (\PP^{n-1}) - T (\PP) \|_{2} & \leq  \| \PP^{n-1} - \PP \|_{2} + C.
\end{align*}
Iterating this inequality, we deduce
\begin{align} \label{eq:proof1}
    \| \PP^n -  \PP \|_{2} \leq  \| \PP^0 - \PP \|_{2} + Cn \leq C(n+1) \,.
\end{align}
Using the interpolation inequality $\| \cdot \|_{H^{3/2}} \leq  \| \cdot \|_{L^2}^{1/4} \| \cdot \|_{H^{2}}^{3/4}$, \eqref{eq:proof-conv-L2} and \eqref{eq:proof1} yield
\begin{align}\label{eq:convh32}
     \| \PP^n - \PP \|_{H^{3/2}} & \leq  \| \PP^n - \PP \|^{\frac{1}{4}} \| \PP^n - \PP \|_{2}^{\frac{3}{4}} \notag 
    \\
    & \leq   \rho^{n/4}\| \PP^0 - \PP \|^{\frac{1}{4}} C^{3/4} (n+1)^{3/4}  \notag
    \\
    & \leq  C^{3/4} \rho^{n/4}  (n+1)^{3/4} \delta^{1/4}
    \,,
\end{align}
where $\|\PP^0 - \PP \|\leq \|\PP^0 - \PP \|_{2}<\delta$. Since $\lim_{n\to \infty} \rho^{n/4} (n+1)^{3/4} =0$, we can choose $\delta >0$ sufficiently small so that 
$C^{3/4} \rho^{n/4}  (n+1)^{3/4}  \leq \delta^{- 1/8} $ for all $n \in \NN$. Therefore, we deduce that 
$$ \| \PP^n - \PP \|_{H^{3/2}} \leq \delta^{1/8} \,, \quad \forall n \in \NN\,.$$
 Thus, $\PP^n - \PP \in K$ for all $n \in \NN$.
In particular, we can iterate the bound \eqref{eq:proof-conv-L2} for all $n \in \NN$ and prove that  
$$\| \PP^n - \PP \| \leq \rho^n \| \PP^0 - \PP \| \qquad \forall n \in \NN \,.$$
The bound in \eqref{eq:convh32} proves the convergence in $H^{3/2}(\mathbb{T}^2, \mathbb{S}^2)$.
\end{proof}

\section{Convergence of the $L^2$ splitting algorithm on bounded convex domains}\label{sec:bounded}
The strategy used to prove convergence of \eqref{eq:projj} on $\mathbb{T}^2$ does not extend directly to general bounded convex domains, as explained in \Cref{remark:projection}. Nevertheless, building on \Cref{lemma:uniform-cont,lemma:frechet-fake}, we establish in this section a local convergence result for convex domains $\overline{\Omega}\subset\mathbb{R}^2$ with $\partial\Omega\in C^4$. As in the periodic setting, control in the $L^\infty$ norm is required to ensure that the iterates remain on the elliptic branch. The main difference is that, instead of proving that $DT(\mathbf{P})$ is contractive in the $H^2$ norm, we study its spectral radius as a bounded operator on $H^2$.

We first assume that \eqref{A1} holds and that
\begin{enumerate}
\item[\refstepcounter{assump}\label{A2}(\theassump)] there exists a unique convex real-valued solution $u\in C^4(\overline{\Omega})$ to \eqref{eq:prob}. In particular, $f\in C^2(\overline{\Omega})$.
\end{enumerate}
We recall that, in the general bounded domain setting, the splitting algorithm \eqref{eq:projj} is defined by 
\begin{equation*}
D^2u^{n} = \Pi_{\mathcal{V}_g}(\mathbf{P}^n), 
\quad 
\mathbf{P}^{n+1} = \Pi_{\mathcal{B}}(D^2u^n),
\quad 
\forall n \ge 0, 
\quad 
\mathbf{P}^0 \text{ given},
\end{equation*}
with $\mathcal{V}_g$ and $\mathcal{B}$ respectively defined by \eqref{d:V-g} and \eqref{d:B}. As in \Cref{sec:setting}, we denote $T = \Pi_\mathcal{B}\circ \Pi_{\mathcal{V}_g}$. We may then write
\begin{align} \label{eq:Taylor-bounded}
T(\PP + \mathbf{H}) = T(\PP) + K(\mathbf{H})+ \cR_{\PP}(\mathbf{H}),
\end{align}
where $K:=\Pi_{\ker(\cof(\PP))}\circ\Pi_{\mathcal{V}_0}$ is the G\^ateaux derivative of $T$ in $\mathbf{P}$, by the same argument as in \Cref{lemma:uniform-cont,lemma:frechet-fake}.

The section is organized as follows. We first prove in \Cref{prop:spectrum} that the spectral radius of $K$, viewed as an operator on $H^2$, is strictly smaller than $1$. We then derive an estimate for $\mathcal{R}_{\mathbf{P}}$ in \Cref{lemma:error}. Finally, in \Cref{thm:bounded}, we establish the local convergence of the splitting algorithm \eqref{eq:projj} by showing that $T$ is locally contractive with respect to a norm equivalent to the $H^2$ norm.

We now study the spectrum of $K$. The proof of the following result relies on elliptic regularity theory, and the precise statement required here is proven in Appendix \ref{app:appendix} for completeness. The smoothness assumption in \eqref{A2} ensures that the available elliptic regularity results are applicable. In the following, we denote by $\mathbb{S}^2_{\mathbb{C}}$ the space of symmetric matrices with entries in $\mathbb{C}$. The scalar products defined in \Cref{sec:setting} are extended in the natural way to functions with values in $\mathbb{C}$. 

\begin{prop} \label{prop:spectrum} Let $\overline{\Omega} \subset \R^2$ be a bounded and convex domain with boundary $\partial\Omega\in C^4$ and \eqref{A1}-\eqref{A2} hold.
There exists $r<1$ such that  the spectrum of the operator $K$ in $\mathcal{L}(H^2,H^2)$ satisfies 
    $$\sigma_{H^2} (K) \subset \{ \lambda: |\lambda |\leq  r\}\,.$$
\end{prop}

\begin{proof}
To prove the result we aim to study the invertibility
\begin{equation}\label{eq:spectral_pb}
    (\lambda\mathrm{Id} - K) \mathbf{X} = \mathbf{F},
\end{equation}
We denote by  
$$ \mathbf{C}_\PP := \frac{\cof (\PP)}{|\cof (\PP)|} \,, \qquad D^2 w := \Pi_{\mathcal{V}_0} \mathbf{X}\,,$$
and  a direct computation shows that \eqref{eq:spectral_pb} is equivalent to 
\begin{align}\label{eq:resolvent-strong-problem}
\begin{cases}
\diver \diver ((\lambda-1)D^2w+(\mathbf{C}_\PP:D^2w)\mathbf{C}_\PP ) =
\diver \diver  \mathbf{F} \quad & \text{on }  \Omega \,,
\\
w=0  &   \text{on }  \partial\Omega\,,
\\
((\lambda-1)D^2w+(\mathbf{C}_\PP:D^2w)\mathbf{C}_\PP): \nu \otimes \nu = \mathbf{F} :  \nu \otimes \nu
&  \text{on }\partial\Omega.
\end{cases}
\end{align} 
As proven in \Cref{lemma:pseudo}, there exists $r<1$ such that, for all $\lambda \in \mathbb C$ with $|\lambda|>r$, for any $\mathbf{F} \in H^2 (\Omega,\mathbb{S}^2_{\mathbb{C}})$, there exists a unique solution $w \in H^4 (\Omega,\mathbb{C}) \cap H^{1}_0 (\Omega,\mathbb{C})$ of \eqref{eq:resolvent-strong-problem} and 
\begin{equation}\label{eq:elliptic-reg}
   \| w \|_{4}\leq C_{\lambda} \| \mathbf{F} \|_{2}. 
\end{equation}

We fix $\lambda \in \mathbb{C}$ with $|\lambda|> r$. We first prove surjectivity of $\lambda\mathrm{Id} -K$. Since $w\in H^4(\Omega,\mathbb{C})$, integration by parts gives the following: for every $v\in H^2(\Omega,\mathbb{C})\cap H^1_0(\Omega,\mathbb{C})$
$$\lambda\int_\Omega D^2w:D^2\overline v=\int_\Omega\left(\mathbf{F}+\Ocof D^2w\right):D^2\overline v.$$
Hence, we deduce that $\mathbf{X}:=\frac1\lambda(\mathbf{F}+\Ocof D^2w)$ satisfies $\Pi_{\mathcal{V}_0} (\mathbf{X}) = D^2 w$. Therefore it holds that $(\lambda \mathrm{Id}-K)\mathbf{X} =\mathbf{F}$ and  
\begin{align} \label{eq:inverse}
    \|\mathbf{X}\|_{2}\le C_\lambda\|\mathbf{F}\|_{2}
\end{align}
follows from \eqref{eq:elliptic-reg} concluding the proof of the surjectivity in $H^2$.
The injectivity is similar. If $(\lambda \mathrm{Id}-K) \mathbf{X}=\mathbf{0}$, the same argument gives a solution $w$ of \eqref{eq:homotopy} with $\mathbf{F}=\mathbf{0}$ and by \eqref{eq:elliptic-reg}, we deduce $w=0$, and hence $\mathbf{X}=\frac1\lambda\Ocof D^2w=\mathbf{0}$. Therefore $\lambda \mathrm{Id}-K$ is also injective. The estimate \eqref{eq:inverse} shows that its inverse is bounded, concluding the proof.
\end{proof}

Now we derive a quadratic bound for $\mathcal{R}_\mathbf{P}$.

\begin{lemma} \label{lemma:error} Let \eqref{A1}-\eqref{A2} hold. Let $\cR_{\PP}$ be defined by \eqref{eq:Taylor-bounded}. Then, there exist $\delta>0 $ and $C>0$ such that, for any $\mathbf{H} \in H^2(\Omega,\SSS^2)$ with $\| \mathbf{H} \|_{H^2}< \delta $, the following holds 
$$ \| \cR_\PP (\mathbf{H})\| \leq C \| \mathbf{H}\|_{2} \| \mathbf{H} \| \,, \qquad  \| \cR_\PP (\mathbf{H})\|_{2} \leq C \| \mathbf{H} \|_{2}^2. $$
\end{lemma}

\begin{proof}
As shown in \Cref{secc:lipschitz}, for every $x \in \Omega$ each $\mathcal B_x = \{ M \in \mathbb S^2: \det M = f(x)\}$ is a smooth embedded hypersurface in $\mathbb{S}^2$, and the local nearest-point projection is smooth, and its derivative on the hypersurface is the orthogonal tangent projection $D_M \Pi_{\mathcal{B}_x} (\PP (x)) = \Pi_{\ker (\cof (\PP (x))) }$. We denote by $\Pi_{\mathcal{B}_x}(M)$ the projection onto $\mathcal{B}_x$ of the matrix $M$. The Taylor expansion  gives
$$
 |\Pi_{\mathcal{B}_x}(\PP(x)+M)-\PP(x)-D_M\Pi_{\mathcal{B}_x}(\PP)M|\leq C|M|^2.
 \qquad \forall |M|\ll 1\,,
$$
where $C$ may depend on $\|\PP\|_{C^2}$.
Taking $M=\Pi_{\mathcal{V}_0} \mathbf{H}(x)$ and using \Cref{lemma:pseudo}
    \begin{align}
        \| \Pi_{\mathcal{V}_0} \mathbf{H}\|_{2} \leq C \|\mathbf{H}\|_{2}\quad \forall\mathbf{H}\in H^2(\Omega,\mathbb{S}^2),
    \end{align}
we obtain by Sobolev embeddings and $L^2$ continuity of $\Pi_{\mathcal{V}_0}$
$$ \|\mathcal R_\PP (\mathbf{H})\|\leq C\|\Pi_{\mathcal{V}_0}( \mathbf{H})\|_{L^\infty}\|\Pi_{\mathcal{V}_0}(\mathbf{H})\|\leq C\|\mathbf{H}\|_{2}\|\mathbf{H}\|.$$
 Differentiating twice in $x$ the Taylor expansion in $M$ and using that $\Pi_{\mathcal{B}_x}$ is $C^2$ in $x$ with norm depending on $\| \PP \|_{C^2}$ and that $H^2$ is an algebra, we  obtain the estimate also in $H^2$ and we conclude the proof.
\end{proof}

Finally, we prove the local convergence of the splitting algorithm \eqref{eq:projj} as a direct consequence of the proposition above. In particular, this proves convergence of $\mathbf{P}^n$ to $\mathbf{P}$ in $H^2$ and, hence, convergence of $u^n$ to $u$ in $H^4(\Omega)$.
\begin{theorem}
\label{thm:bounded}
Let $\overline{\Omega} \subset \R^2$ be a bounded and convex domain with boundary $\partial\Omega\in C^4$. Under assumptions \eqref{A1}-\eqref{A2}, there exist constants $\delta>0$, $C\geq1$, and $\rho\in(0,1)$ such that, for any initial guess $\PP^0\in H^2(\Omega,\mathbb S^2)$ with $\|\PP^0-\PP\|_{2}<\delta$, the sequence $(\mathbf{P}^n)_{n\geq 0}$ generated by \eqref{eq:projj} satisfies
\begin{align}
 \|\PP^n-\PP\|_{2}
 &\leq C\rho^n\|\PP^0-\PP\|_{2},\label{bd:H2-convergence}
\end{align}
In particular, $\PP^n\to\PP$ in $H^s(\Omega,\mathbb S^2)$ for any $0\leq s\leq2$. 
\end{theorem}

\begin{proof}
By \Cref{prop:spectrum}, the spectral radius of $K$ on $H^2$ is at most $r<1$. Choose
$$
 r<r_1<r_2<1\,.
$$
Gelfand's spectral-radius formula, see for example \cite[Chapter VIII.2, Theorem 3]{yosida}, implies that there exists $C>0$ such that 
\begin{equation}\label{bd:powers}
 \|K^n \|_{\cL(H^2,H^2)}\leq Cr_1^n \,, \qquad \forall n.
\end{equation}
On the real space $H^2(\Omega,\mathbb S^2)$ define
\begin{equation}\label{bd:adapted-norm}
 \|\mathbf{H}\|_*:=\sup_{n\geq0}r_1^{-n}\|K^n \mathbf{H}\|_{2}.
\end{equation}
It is an equivalent norm to $H^2$, and $\|K (\mathbf{H})\|_*\leq r_1 \|\mathbf{H}\|_*$. By \Cref{lemma:error} we get
$$ \|T (\PP+ \mathbf{H}) - \PP\|_* \leq r_1 \|\mathbf{H}\|_*+ C\|\mathbf{H}\|_*^2.$$
Choose $\varepsilon>0$ sufficiently small so that $T$ is defined on
$\PP + \{\mathbf{H}:\|\mathbf{H}\|_*\leq\varepsilon\}$ and
$C\varepsilon\leq r_2- r_1$. Then
\begin{equation}\label{bd:nonlinear-contraction}
 \|T(\PP+ \mathbf{H}) - \PP \|_*\leq r_2\|\mathbf{H}\|_*\,.
\end{equation}
In particular, $ T(\PP + \mathbf{H}) - \PP$ is invariant on the ball $\{\mathbf{H}:\|\mathbf{H}\|_*\leq\varepsilon\}$.
Taking $\delta<\varepsilon/C$, we can estimate 
\begin{align}
    \| \PP^n-\PP\|_* = \| T(\PP+ (\PP^{n-1} - \PP)) - \PP \|_* \leq r_2 \| \PP^{n-1} - \PP \|_* \leq r_2^n \| \PP^{0} - \PP \|_* \,.
\end{align}

Hence, we conclude the convergence in $H^2$ by the equivalence of the norms. 
\end{proof}

\section{Numerical examples}\label{sec:num}
In this section we present numerical experiments on $\mathbb{T}^2$. Although the numerical discretization is not the focus of this work, we include these results because, to the best of our knowledge, the proposed method has not previously been tested in the periodic setting. The purpose of this section is solely to assess the validity of \Cref{thm:main}, rather than to discuss the numerical discretization of \eqref{eq:splitting}. We refer to \cite{peruso} for the details of the numerical approximation of \eqref{eq:biharmonic} and \eqref{eq:firstmin}. Specifically, we employ $\mathbb{P}_1$ finite elements to approximate the biharmonic problem \eqref{eq:biharmonic}, yielding an approximation $u_h$ (and $\Delta u_h$ via the decoupling procedure described in \cite{peruso}), where $h$ denotes the mesh size. By post-processing $\nabla u_h$ with the recovery operator $G_h$, we obtain a discrete approximation of the Hessian of $u$, denoted by $D_h^2u_h$. The minimization problem \eqref{eq:firstmin} is solved as in \cite{glowinski}. 

We represent $\mathbb{T}^2$ as $[0,1]^2$ endowed with periodic boundary conditions. We consider a manufactured solution given by
\[
u(x,y) = \varepsilon \sin(2\pi x)\sin(2\pi y).
\]
For this choice, the eigenvalues of $\mathbf{I} + \mathbf{P}=\mathbf{I} + D^2u$ belong to $[1 - 4\pi^2\eps, 1 + 4\pi^2\eps]$ and therefore \eqref{A1} holds provided that $\eps < 1/4\pi^2\approx 0.025$. 

In \Cref{fig:small} we report the results for $\eps = 0.002$. In particular, \Cref{fig:small1} shows the error in $u$ with respect to the $L^2$ norm, namely $\|u-u_h^n\|_{L^2}$, as a function of the iteration index $n$ for several mesh sizes $h$. After approximately $10$ iterations the error reaches a plateau. In \Cref{fig:small2} we plot the corresponding $L^2$ error for the Hessian reconstruction, $\|D^2u-D_h^2u_h^n\|_{L^2}$, and observe the same behavior, in agreement with \Cref{thm:main}. Finally, \Cref{fig:small3} reports the errors in several norms as functions of the mesh size $h$, indicating convergence also with respect to mesh refinement. The analogous plots for $\eps = 0.02$ are shown in \Cref{fig:big}, leading to the same qualitative conclusions. We additionally observe that convergence in $n$ is significantly slower when $\eps=0.02$ compared to $\eps=0.002$. This is consistent with the fact that, as $\eps$ increases, the smallest eigenvalue $\nu_1$ of $\mathbf{I}+D^2u$ ($\mathbf{I} + \PP$) decreases and the ratio $\kappa=\nu_2/\nu_1$ increases. 

\begin{figure}[tbp]
\centering
\begin{subfigure}[t]{0.325\linewidth}
\centering
\includegraphics[width=0.98\linewidth]{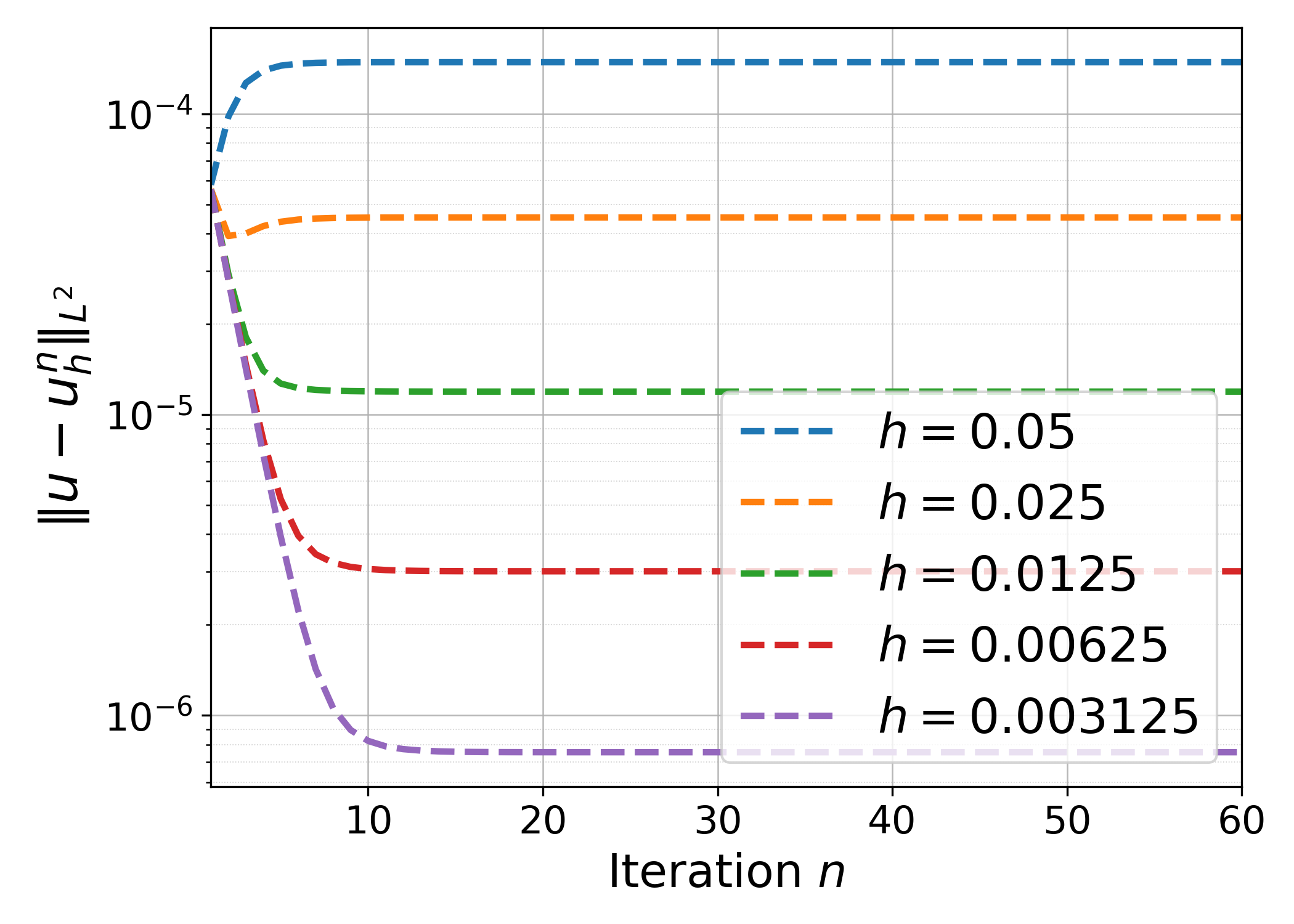}
\caption{$L^2$ error vs. iteration.}
\label{fig:small1}
\end{subfigure}
\begin{subfigure}[t]{0.325\linewidth}
\centering
\includegraphics[width=0.98\linewidth]{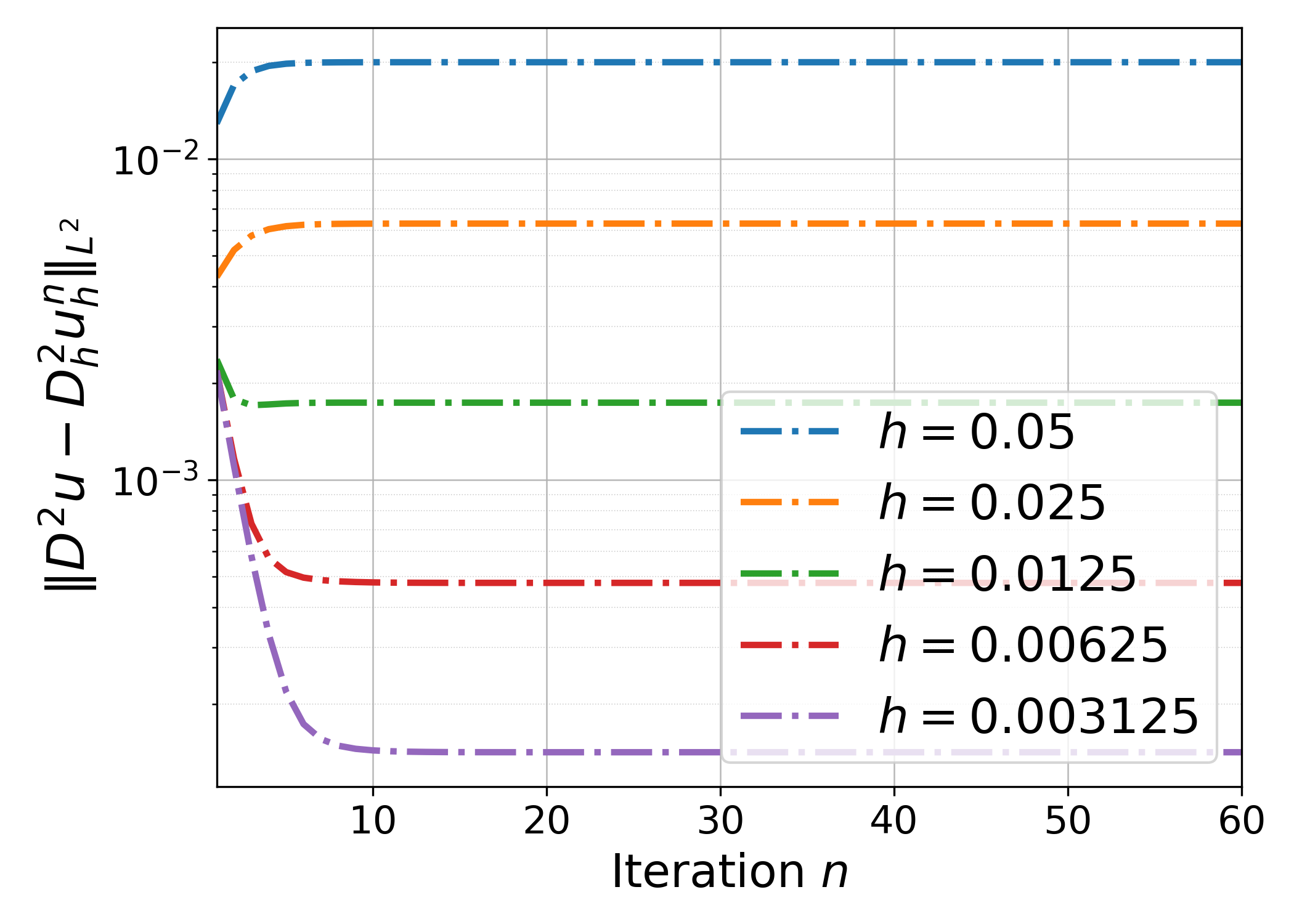}
\caption{$H^2$ error vs. iteration.}
\label{fig:small2}
\end{subfigure}
\begin{subfigure}[t]{0.325\linewidth}
\centering
\includegraphics[width=0.98\linewidth]{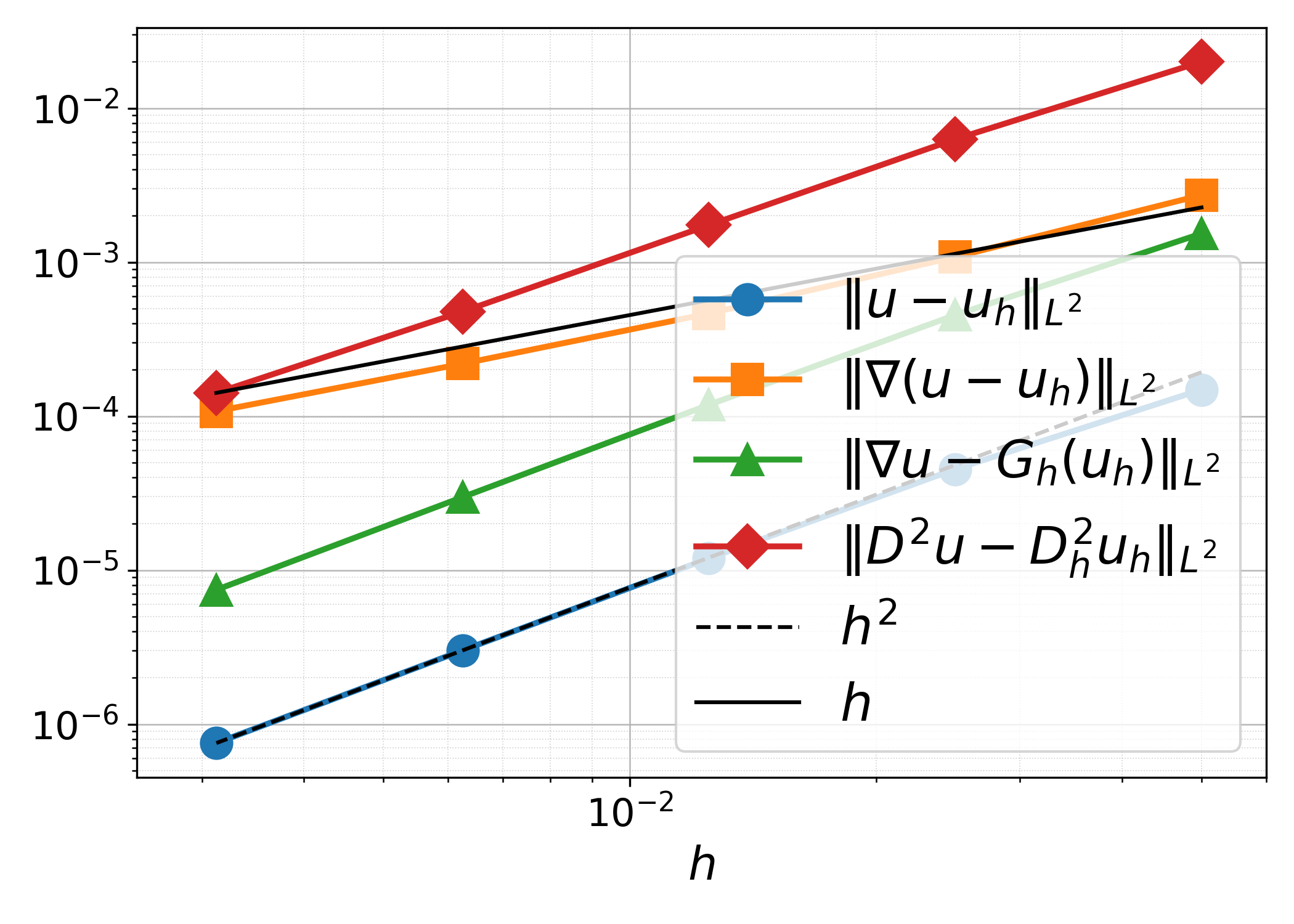}
\caption{Errors vs. $h$.}
\label{fig:small3}
\end{subfigure}
\caption{Results for $\eps = 0.002$.}
\label{fig:small}
\end{figure}

\begin{figure}[tbp]
\centering
\begin{subfigure}[t]{0.325\linewidth}
\centering
\includegraphics[width=0.98\linewidth]{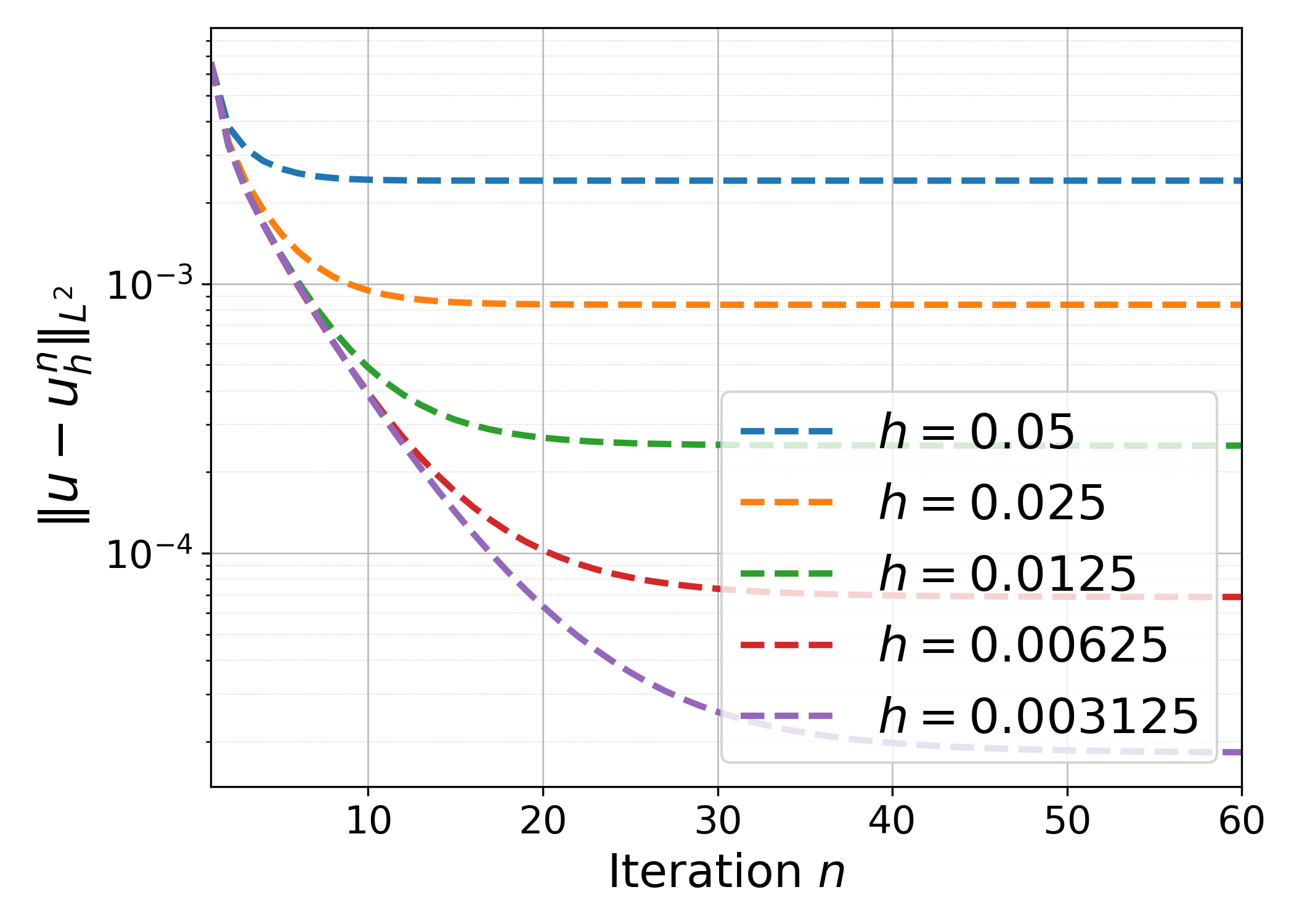}
\caption{$L^2$ error vs. iteration.}
\end{subfigure}
\begin{subfigure}[t]{0.325\linewidth}
\centering
\includegraphics[width=0.98\linewidth]{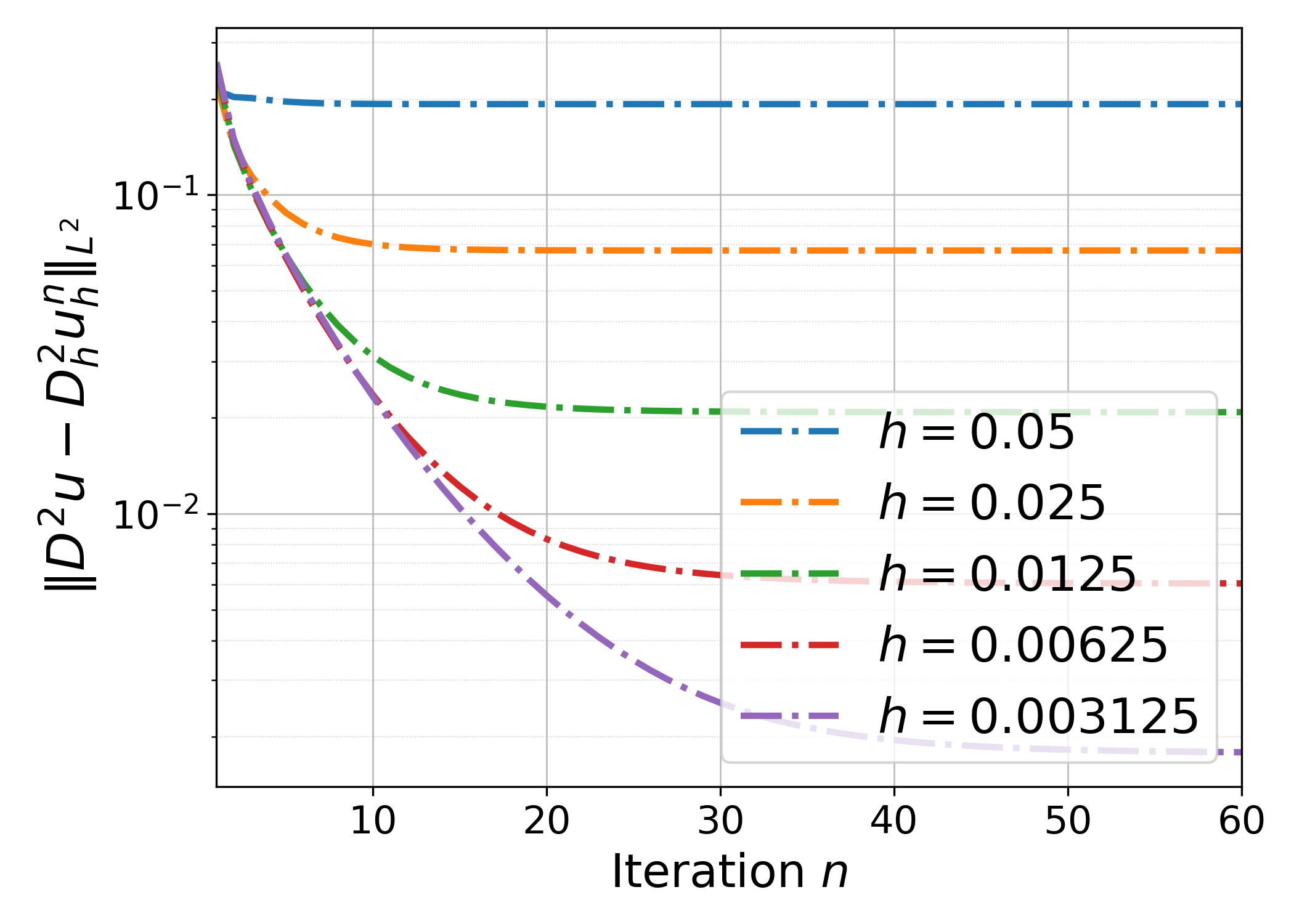}
\caption{$H^2$ error vs. iteration.}
\end{subfigure}
\begin{subfigure}[t]{0.325\linewidth}
\centering
\includegraphics[width=0.98\linewidth]{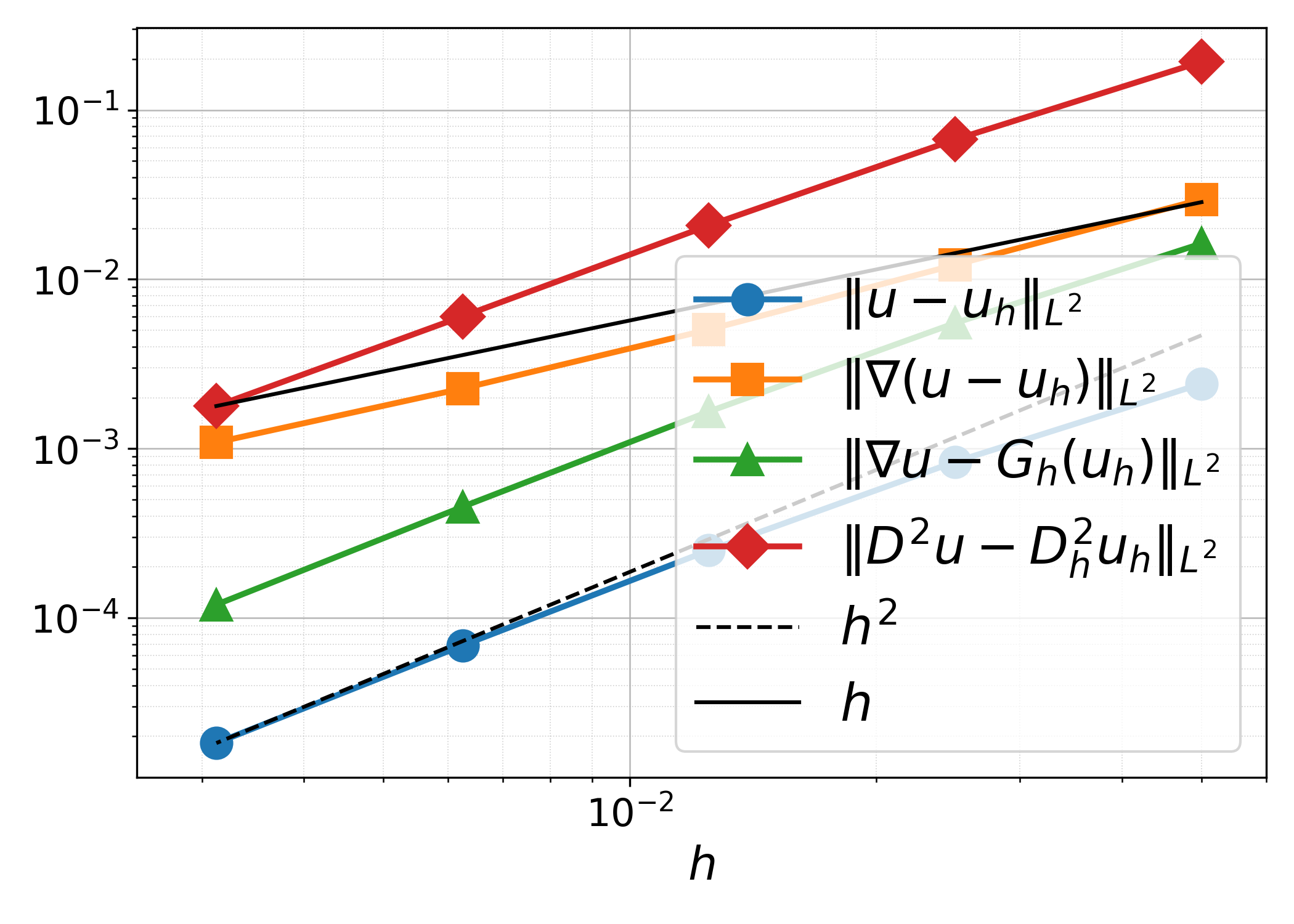}
\caption{Errors vs. $h$.}
\end{subfigure}
\caption{Results for $\eps = 0.02$.}
\label{fig:big}
\end{figure}

\section{Conclusions}
We analyzed the nonlinear least-squares splitting method of \cite{dean,caboussat,peruso} for the Monge-Amp\`ere equation by recasting it as alternating projections in Sobolev spaces. This viewpoint yields a transparent fixed-point framework and highlights the specific difficulties of the original  formulation at low regularity.

Our main result establishes local convergence of the splitting method on sufficiently smooth bounded convex domains. Assuming \(f\in C^2(\overline{\Omega})\), \(u\in C^4(\overline{\Omega})\), and uniform ellipticity of \(D^2u\), we show that sufficiently accurate initial data yield convergence to \(u\) in \(H^4\). To the best of our knowledge, this provides the first rigorous convergence theory for the underlying splitting method for the Monge-Amp\`ere equation. As an intermediate step, we prove local linear convergence on \(\mathbb{T}^2\). We also establish convergence for higher-regularity variants based on \(H^m\)-projections, \(m\ge2\), on bounded convex domains. These results may extend to higher dimensions, although they are mainly of theoretical interest because of their computational cost.

We believe that the argument may extend to other elliptic PDEs depending only on the eigenvalues of $D^2u$ and to other boundary conditions, under analogous regularity and ellipticity assumptions. 

An open problem is to establish the well-posedness and convergence of the splitting method in higher dimensions, for which numerical evidence is available \cite{dimitrios}. Another natural direction for future work is to extend the analysis to other boundary conditions, which would broaden the range of potential applications.

\section*{Acknowledgments}
The authors are grateful to Alexandre Caboussat and Marco Picasso for  proposing the problem.
MS acknowledges support from the Chapman Fellowship at Imperial College London.

\appendix
\section{Well-posedness and regularity of \eqref{eq:resolvent-strong-problem}}\label{app:appendix}
\begin{lemma} \label{lemma:pseudo}
Let $\overline{\Omega} \subset \R^2$ be a bounded and convex domain with boundary $\partial\Omega\in C^4$ and \eqref{A1}-\eqref{A2} hold. There exists $r<1$ such that, for all $\lambda \in \mathbb C$ with $|\lambda|>r$, for any $\mathbf{F} \in H^2 (\Omega,\mathbb{S}^2_{\mathbb{C}})$, there exists a unique weak solution $w \in H^4 (\Omega,\mathbb{C}) \cap H^{1}_0 (\Omega,\mathbb{C})$ of \eqref{eq:resolvent-strong-problem} and 
$$ \| w \|_{4}\leq C_{\lambda} \| \mathbf{F} \|_{2}.$$
\end{lemma}
Regularity results analogous to \Cref{lemma:pseudo} are standard for the biharmonic operator $\Delta^2$; see \cite{gazzola}. However, we have not found a proof covering the precise fourth-order boundary value problem in \eqref{eq:resolvent-strong-problem}, so we provide one here. To obtain the a priori estimate, we verify the Lopatinski\u{\i}-\v{S}apiro complementing condition and apply \cite[Theorem 13.1]{Wloka}. We then prove the required regularity by the method of continuity (see \cite[Theorem 5.2]{gilbarg2001elliptic}), using \cite[Theorem 2.20]{gazzola} as starting point.
\begin{proof}[Proof of \Cref{lemma:pseudo}]
We introduce the parameter $t\in [0,1]$ and for any $t$ we study 
\begin{align}\label{eq:homotopy}
\begin{cases}
\diver \diver ((\lambda-t)D^2w+ t(\mathbf{C}_\PP:D^2w)\mathbf{C}_\PP ) =
\diver \diver  \mathbf{F} \quad & \text{on }  \Omega \,,
\\
w=0  &   \text{on }  \partial\Omega\,,
\\
((\lambda-t)D^2w+ t(\mathbf{C}_\PP:D^2w)\mathbf{C}_\PP): \nu \otimes \nu = \mathbf{F} :  \nu \otimes \nu
&  \text{on }\partial\Omega\,,
\end{cases}
\end{align} 
and we introduce the short-hand notation  
\begin{align*}
    L_{\lambda,t}w :=\diver \diver ((\lambda-t)D^2w+ t(\mathbf{C}_\PP:D^2w)\mathbf{C}_\PP ),
 \\
 M_{\lambda,t}w
 :=((\lambda-t)D^2w+ t(\mathbf{C}_\PP:D^2w)\mathbf{C}_\PP): \nu \otimes \nu\,.
\end{align*}
For $t=0$, the differential operator in \eqref{eq:homotopy} becomes $\lambda\Delta^2 w$ and by  \cite[Theorem 2.20]{gazzola} there exists a unique weak solution $w\in H^4 (\Omega,\mathbb{C}) \cap H^1_0(\Omega,\mathbb{C})$ such that 
$$ \| w \|_{4} \leq C_\lambda \| \mathbf{F} \|_{2} \,.$$

We now aim to apply the method of continuity for $t\in [0,1]$ to deduce the same estimate for $t=1$.
We first check existence and uniqueness of a weak solution $w\in H^2(\Omega,\mathbb{C})\cap H^1_0(\Omega,\mathbb{C})$ of \eqref{eq:homotopy} and later upgrade the regularity to $H^4(\Omega,\mathbb{C})$ with a quantitative bound. In particular, we consider the problem of finding $w\in H^2(\Omega,\mathbb{C})\cap H^1_0(\Omega,\mathbb{C})$ such that $a_{\lambda,t}(w,v) = \ell(v)$ for any $v\in H^2(\Omega,\mathbb{C})\cap H^1_0(\Omega,\mathbb{C})$, where 
$$a_{\lambda,t}(w,v):=(\lambda-t)\int_\Omega D^2w:D^2\overline v+t\int_\Omega(\mathbf{C}_\PP:D^2w)(\mathbf{C}_\PP:D^2\overline v),$$
$$\ell(v):=\int_\Omega \mathbf{F}:D^2\overline v.$$
We now show that there exists $r\in (0,1)$ such that $\frac{\overline\lambda}{|\lambda|}a_{\lambda,t}$ is coercive for any $|\lambda|>r$. By definition we have
$$\frac{\overline\lambda}{|\lambda|}a_{\lambda,t}(v,v) = |\lambda| \|D^2v \|^2 +t\frac{\overline\lambda}{|\lambda|}\bigl(\|C_{\PP}:D^2v\|^2-\|D^2v\|^2\bigr),$$
Now, by assumption, $\PP \in C^2$ and uniformly positive definite on $\overline\Omega$. Hence, we have that 
$$\delta:=\|I- {\rm tr} (\mathbf{C}_\PP) \mathbf{C}_\PP \|_{L^\infty(\Omega)}<1 \,.$$
Using this and the Miranda-Talenti inequality we obtain for any $v\in H^2(\Omega,\mathbb C)\cap H^1_0(\Omega,\mathbb C)$
$$\|D^2v\|\leq \|\Delta v\| = \|I:D^2v\|\leq \| {\rm tr} (\mathbf{C}_\PP)\|_{L^\infty}\|\mathbf{C}_\PP:D^2v\|+\delta\|D^2v\|.$$
Then there exists $c>0$ depending on $\delta$ and ${\rm tr} (\mathbf{C}_\PP)$ such that  
\begin{equation}\label{eq:diseq-c}
- (1-c^2)\|D^2v \|^2 \leq \|C_{\PP}:D^2v \|^2- \|D^2v \|^2 \leq 0.    
\end{equation}
setting $r:= 1-c^2$ we have 
\begin{equation*}
{\rm Re} \left(
\frac{\overline\lambda}{|\lambda|}a_{\lambda,t}(v,v)
\right)
\geq (|\lambda|-r)\|D^2v\|^2.
\end{equation*}
Thus, for every $|\lambda |> r$, we have that for every $\mathbf{F}\in H^2(\Omega,\mathbb{S}^2_{\mathbb C})$, the Lax-Milgram theorem applied to $\frac{\overline\lambda}{|\lambda|}a_{\lambda,t}(w,v) = \frac{\overline\lambda}{|\lambda|}\ell(v)$ yields a unique solution $w\in H^2(\Omega,\mathbb C)\cap H^1_0(\Omega,\mathbb C)$ satisfying
\begin{align}
\|w\|_{2} & \leq C_\lambda (\| \diver \diver (\mathbf{F})\| + \| \mathbf{F} : \nu \otimes \nu \|) \notag
\\
& =C_\lambda (\| L_{\lambda, t} w\| + \| M_{\lambda, t} w \|)\,. \label{eq:weak-H2-estimate}
\end{align}

We now aim to verify the ellipticity condition on $L_{\lambda, t}$ and the Lopatinski\u{\i}-\v{S}apiro complementing  assumption in order to apply \cite[Theorem 13.1]{Wloka} and get the following a priori estimate: 
\begin{equation*}
  \|w\|_{4}\leq C_{\lambda,t}\left ( \|L_{\lambda, t} w\| + \|M_{\lambda,t}w\|_{H^{3/2}(\partial\Omega)} + \| w \|_{3} \right )  \quad \forall w\in H^4(\Omega,\mathbb{C})\cap H^1_0(\Omega,\mathbb{C})\,,
  \end{equation*}
  and by interpolation inequality, Young-inequality  and \eqref{eq:weak-H2-estimate} we deduce 
\begin{equation}\label{eq:apriori_wloka}
  \|w\|_{4}\leq C_{\lambda,t}\left ( \|L_{\lambda, t} w\| + \|M_{\lambda,t}w\|_{H^{3/2}(\partial\Omega)}  \right )  \quad \forall w\in H^4(\Omega,\mathbb{C})\cap H^1_0(\Omega,\mathbb{C})\,,
  \end{equation}
The principal symbol of $L_{\lambda,t}$ is
$$
p_{\lambda,t}(x,\xi) =|\xi|^4 \left[ \lambda-t\left( 1-\frac{(\mathbf{C}_\PP(x):\xi\otimes\xi)^2}{|\xi|^4} \right) \right].
$$
The uniform ellipticity of $\mathbf{C}_\PP$ implies that, up to take a slightly bigger $r<1$, we have
$$
|p_{\lambda,t}(x,\xi)| \geq (|\lambda|-r)|\xi|^4 \,, \qquad \forall t \in [0,1]\,,
$$
and the ellipticity condition on $L_{\lambda, t}$ follows.

We now verify the Lopatinski\u{\i}-\v{S}apiro complementing condition. Fix $x_0\in\partial\Omega$ and choose orthonormal coordinates $(y,s)$ such that $y$ is tangential to $\partial\Omega$ at $x_0$ and $s$ is normal to it. 
For a nonzero tangential frequency $\xi\in\mathbb R\setminus\{0\}$,
consider a decaying solution of the frozen homogeneous problem of the form
$$
u(y,s)=e^{i\xi y}z(s),\qquad s>0.
$$
The Lopatinski\u{\i}-\v{S}apiro complementing condition requires that the only solution is the trivial one $z\equiv 0$. We introduce
$$
H_\xi z:=
\begin{pmatrix}
-\xi^2z & i\xi z'\\
i\xi z' & z''
\end{pmatrix} \,, \qquad A_\xi z:= (\lambda-t)H_\xi z +t(C_{\PP}(x_0):H_\xi z)C_{\PP}(x_0)\,.
$$
The equations \eqref{eq:homotopy} with  homogeneous boundary conditions $w|_{\partial\Omega} =0 \,,
M_{\lambda,t}w|_{\partial\Omega} =0$ and  frozen coefficients evaluated at $x_0$  read as 
\begin{equation}\label{eq:frozen-complementing}
\begin{cases}
-\xi^2(A_\xi z)_{11}
+2i\xi (A_\xi z)'_{12}
+(A_\xi z)''_{22}=0,
& s>0,\\
z(0)=0,\\
(A_\xi z)_{22}(0)=0.
\end{cases}
\end{equation}

Multiplying the equation in \eqref{eq:frozen-complementing} by
$\overline z$ and integrating twice by parts, all the boundary terms vanish since $z(0)=0$, $(A_\xi z)_{22}(0)=0$, and $z$ is decaying. Therefore,
\begin{equation}\label{eq:integrand-zero}
0 = \int_0^\infty A_\xi z: \overline{H_\xi z}\,ds = (\lambda-t)\|H_\xi z\|^2 +t\|C_{\PP}(x_0):H_\xi z\|^2.    
\end{equation}
Moreover, using $z(0)=0$ and integration by parts,
\begin{align*}
\|\tr(H_\xi z)\|^2&=\|z''-\xi^2z\|^2\\
&=\|z''\|^2+2\xi^2\|z'\|^2+\xi^4\|z\|^2\\
&=\|H_\xi z\|^2.
\end{align*}
Hence the same argument leading to \eqref{eq:diseq-c} yields
$$
-r\|H_\xi z\|^2 \leq \|C_{\PP}(x_0):H_\xi z\|^2-\|H_\xi z\|^2 \leq0.
$$
Consequently, multiplying \eqref{eq:integrand-zero} by $\overline\lambda/|\lambda|$ gives
$$
0 \geq (|\lambda|-r)\|H_\xi z\|^2.
$$
Since $|\lambda|>r$, it follows that $H_\xi z=\mathbf{0}$, which yields $z\equiv0$.

We now define the operator $\mathcal P_{\lambda, t}: H^4(\Omega,\mathbb C)\cap H^1_0(\Omega,\mathbb C)
 \to L^2(\Omega,\mathbb C)\times  H^{3/2}(\partial\Omega;\mathbb C),$ as 
\begin{equation}\label{bd:strong-map}
 \cP_{\lambda, t}: w\longmapsto(L_{\lambda,t}w,M_{\lambda,t}w).
\end{equation}
For the last step, we prove that, for $\mathbf{F}\in H^2(\Omega,\mathbb{S}^2_\mathbb{C})$, the unique weak solution $w$ to \eqref{eq:homotopy} belongs to $H^4(\Omega,\mathbb{C})$. For $t=0$, the required property follows from \cite[Theorem 2.20]{gazzola}. To extend this property to all $t\in[0,1]$, we apply the method of continuity; see, for instance, \cite[Theorem 5.2]{gilbarg2001elliptic}. First, we observe that $t\mapsto\mathcal{P}_{\lambda,t}$ is continuous in operator norm thanks to the relation
$$
(\cP_{\lambda,t}-\cP_{\lambda,s})w
= (t-s) \begin{pmatrix}
\diver\diver \left(
(\mathbf{C}_\PP:D^2w)\mathbf{C}_\PP-D^2w \right)
\\
\left(
(\mathbf{C}_\PP:D^2w)\mathbf{C}_\PP-D^2w
\right):(\nu\otimes\nu)
\end{pmatrix}\,.
$$
Then, we observe that the constant $C_{\lambda, t}$ in \eqref{eq:apriori_wloka} can be re-defined as 
$$ \frac{1}{C_{\lambda, t}}:= \inf_{\substack{w\in H^4(\Omega,\mathbb{C})\cap H^1_0(\Omega,\mathbb{C})\\ \|w\|_{4}=1}}\{\|L_{\lambda, t} w\| + \|M_{\lambda,t}w\|_{H^{3/2}(\partial\Omega)}\}$$
which is a continuous function on a compact and therefore $\frac{1}{C_\lambda} = \min_{t\in [0,1]} \frac{1}{C_{\lambda, t}} >0$. Thus, \eqref{eq:apriori_wloka} holds with the constant $C_\lambda $ independent of $t$. The method of continuity therefore implies that $\cP_{\lambda, 1}$ is invertible and therefore for any $\mathbf{F} \in H^2(\Omega(\Omega,\mathbb{C})) $ there exists a unique solution $w\in H^4(\Omega(\Omega,\mathbb{C}))$ to \eqref{eq:homotopy} satisfying 
$$ \| w \|_{H^4} \leq C_\lambda (\| \diver \diver (\mathbf{F}) \| + \| \mathbf{F} \|_{H^{3/2} (\partial \Omega)}) $$
 concluding the proof by the trace theorem.

\end{proof}

\bibliographystyle{unsrt}
\bibliography{biblio}

\end{document}